\documentclass[11pt,reqno]{amsart}

\usepackage{amssymb}
\usepackage{amscd}
\usepackage{amsfonts}
\usepackage{mathrsfs}
\usepackage{setspace}
\usepackage{version}
\usepackage{mathtools}
\usepackage{enumerate}
\usepackage[english]{babel}

\newtheorem{theorem}{Theorem}[section]
\newtheorem{lemma}[theorem]{Lemma}
\newtheorem{proposition}[theorem]{Proposition}
\newtheorem{corollary}[theorem]{Corollary}
\theoremstyle{definition}
\newtheorem{definition}[theorem]{Definition}
\newtheorem{example}[theorem]{Example}

\theoremstyle{remark}
\newtheorem{remark}[theorem]{Remark}
\numberwithin{equation}{section}

\renewcommand{\AA}{\mathcal{A}}
\newcommand{\BB}{\mathcal{B}}

\newcommand{\EE}{\mathcal{E}}
\newcommand{\FF}{\mathcal{F}}

\newcommand{\LL}{\mathcal{L}}

\newcommand{\PP}{\mathcal{P}}
\newcommand{\QQ}{\mathcal{Q}}

\renewcommand{\SS}{\mathscr{S}}

\newcommand{\field}[1]{\mathbb{#1}}

\newcommand{\R}{\field{R}}
\newcommand{\N}{\field{N}}
\newcommand{\Z}{\field{Z}}
\newcommand{\Q}{\field{Q}}

\newcommand{\E}{\field{E}}
\newcommand{\T}{\field{T}}
\renewcommand{\P}{\field{P}}

\newcommand{\al}{\alpha}

\newcommand{\de}{\delta}
\newcommand{\ep}{\varepsilon}

\newcommand{\la}{\lambda}

\newcommand{\si}{\sigma}

\newcommand{\om}{\omega}

\newcommand{\De}{\Delta}

\newcommand{\La}{\Lambda}
\newcommand{\Si}{\Sigma}

\newcommand{\Om}{\Omega}

\newcommand{\BUC}{\text{\rm{BUC}}}

\newcommand{\sm}{\setminus}
\newcommand{\es}{\emptyset}

\newcommand{\ca}{\mathop{\text{\upshape{ca}}}\nolimits}

\begin{document}

\title[A semigroup approach to nonlinear L\'evy processes]{A semigroup approach to nonlinear L\'evy processes}% and their PDE\MakeLowercase{s}}

\author{Robert Denk}
\address{Department of Mathematics and Statistics, University of Konstanz}
\email{Robert.Denk@uni-konstanz.de}

\author{Michael Kupper}
\address{Department of Mathematics and Statistics, University of Konstanz}
\email{kupper@uni-konstanz.de}

\author{Max Nendel}
\address{Center for Mathematical Economics, Bielefeld University}
\email{max.nendel@uni-bielefeld.de}

\date{\today}

\thanks{We thank two anonymous referees for their helpful comments. Financial support through the German Research Foundation via CRC 1283 ``Taming Uncertainty'' is gratefully acknowledged.}

%\subjclass[2010]{}

\begin{abstract}
We study the relation between L\'evy processes under nonlinear expectations, nonlinear semigroups and fully nonlinear PDEs. First, we establish a one-to-one relation between nonlinear L\'evy processes and nonlinear Markovian convolution semigroups. Second, we provide a condition on a family of infinitesimal generators $(A_\lambda)_{\lambda\in\Lambda}$ of linear L\'evy processes which guarantees the existence of a nonlinear L\'evy process such that the corresponding nonlinear Markovian convolution semigroup
is a viscosity solution of the fully nonlinear PDE $\partial_t u=\sup_{\la\in \La} A_\la u$. The results are illustrated with several examples.


\end{abstract}

\maketitle

\setcounter{tocdepth}{1}

\section{Introduction}
In this paper we study the relation between L\'evy processes under nonlinear expectations, nonlinear semigroups and fully nonlinear PDEs.
Let $(X_t)_{t\geq 0}$ be an $\R^d$-valued L\'evy process on a probability space $(\Om,\FF,\P)$. Then
\[
 (S(t)f)(x):=\E(f(x+X_t))
\]
 for $f\in \BUC(\R^d)$, $t\geq 0$ and $x\in \R^d$, defines a strongly continuous semigroup on the space of bounded and uniformly continuous functions $\BUC(\R^d)$ whose infinitesimal generator  $A\colon D(A)\subset\BUC(\R^d)\to\BUC(\R^d)$ is given by an integro-differential operator which is uniquely determined by a L\'evy triplet, see Applebaum~\cite{MR2512800} or Sato~\cite{MR1739520}.
Moreover, $(S(t)f)_{t\ge 0}$ is the solution
 of the abstract Cauchy problem
\[\partial_t u(t) = Au(t), \quad t>0,\]
with $u(0)=f\in \BUC(\R^d)$. For a detailed discussion on operator semigroups we refer to Pazy\ \cite{MR710486} or Engel and Nagel\ \cite{MR2229872}.

 We first extend the well-known relation between L\'evy processes and Markovian convolution semigroups of probability measures to a nonlinear setting. Nonlinear L\'evy processes were introduced in \cite{PengHu} as c\`adl\`ag
processes with stationary and independent increments under a sublinear expectation.
The G-Brownian motion due to \cite{MR2143645,PengG} is a special case of a nonlinear L\'evy process, see also Dolinsky et al.~\cite{MR2868935} or Denis et al.~\cite{MR2754968}. For an introduction to nonlinear expectations we refer to \cite{Peng}. Since we do not require any path regularity we call a process $(X_t)_{t\ge 0}$ an $\mathcal{E}$-L\'evy process with values in an abelian group $G$, if it has stationary and independent increments and  $X_t\to X_0$ in distribution as $t\searrow 0$ with respect to a convex expectation $\mathcal{E}$. We then provide a relation between $\mathcal{E}$-L\'evy processes and convex Markovian convolution semigroups, i.e.~strongly continuous
semigroups $(\SS(t))_{t\ge 0}$ on $\BUC(G)$ such that each $\SS(t)$ is a translation invariant convex kernel.
The proof relies on a version of Kolmogorov's extension theorem for nonlinear expectations elaborated in \cite{DKN}.

Our main focus lies on the construction of $\EE$-L\'evy processes with nonlinear generators. We start with an arbitrary family $\Lambda$
of generators \[A_\la\colon D(A_\la)\subset \BUC(G)\to \BUC(G)\] of Markovian convolution semigroups $S_\la=(S_\la(t))_{t\geq 0}$ of linear operators. We then construct the smallest Markovian convolution semigroup $(\SS(t))_{t\ge 0}$ which dominates each $(S_\la(t))_{t\geq 0}$. The corresponding $\EE$-L\'evy process can be viewed as a process with independent increments whose distribution is uncertain, i.e.~any distrubution of the increments associated to $(S_\lambda)_{\lambda\in\Lambda}$ is taken into account. We basically follow an idea by Nisio \cite{Nisio} in order to construct a sublinear Markovian convolution semigroup which results from a given family of linear Markovian convolution semigroups by constant optimization. In \cite{Nisio} Nisio considers strongly continuous semigroups on the space of all bounded measurable functions. However, by a theorem of Lotz (see e.g. \cite[Corollary 4.3.19]{MR2798103}), any strongly continuous semigroup on the space of bounded measurable functions already has a bounded generator, which is not suitable for most applications. We therefore modify Nisio's construction to the space $\BUC(G)$. Under the condition
that the subspace
 \begin{equation}\label{as:main} \bigg\{ f\in \bigcap_{\la\in \La}D(A_\la) : \big\{A_\la f\colon \la\in \La\big\}\;\text{is bounded and uniformly equicontinuous}\bigg\}
 \end{equation}
 is dense in $\BUC(G)$ we construct a (strongly continuous) Markovian convolution semigroup $(\SS(t))_{t\ge 0}$
 on $\BUC(G)$ with corresponding $\EE$-L\'evy process $(X_t)_{t\ge 0}$ on a sublinear expectation space $(\Omega,\mathcal{F},\EE)$ such that
 \[u(t,x):=(\SS(t)f)(x)=\EE(f(x+X_t)),\quad  t\geq 0,\, x\in G,\]
 is a viscosity solution of the fully nonlinear PDE
 \[\partial_t u = \sup_{\la\in \La} A_\la u  \quad \mbox{on } (0,\infty)\times G
\]
with $u(0)=f$ for all $f\in \BUC(G)$. In particular, the generator of the $\EE$-L\'evy process $(X_t)_{t\ge 0}$ is given by $\sup_{\la\in \La} A_\la$. Here, we use a slightly different notion of viscosity solution which fits to the semigroup setting. However, in many cases, particularly for the classical case $G=\R^d$, this leads to the same class or an even larger class of test functions. We refer to Crandall et al. \cite{MR1118699} for the classical definition and a detailed discussion of viscosity solutions. Moreover, we give a condition on the generators $(A_\lambda)_{\lambda\in\Lambda}$
which guarantees that the corresponding $\EE$-L\'evy process is tight, or equivalently each $\SS(t)$ is continuous from above. Throughout, the state space is an abelian group, which gives the opportunity to consider certain classes of cylindrical $G$-Wiener Processes as an infinite dimensional extension of the $G$-Brownian motion, or nonlinear L\'evy processes on the $d$-dimensional Torus.

Nonlinear $\R^d$-valued L\'evy processes have first been introduced in Hu and Peng \cite{PengHu}, where $G$-L\'evy processes with a decomposition $X=X^c+X^d$ into a continuous and a jump part are considered. Under the assumption that $X^c$ is a $G$-Brownian motion and $\EE(|X^d_t|)\le ct$ for some constant $c$, it is shown that $u(t,x)=\EE(f(x+X_t))$ is a viscosity solution of $\partial_t u(t,x)-G(u(t,x+\cdot)-u(t,x))=0$ and $u(0)=f$, where $G(\varphi(\cdot)):=\lim_{h\downarrow 0} \frac{1}{h}\EE(\varphi(X_h))$. The function $G$ is shown to have a L\'evy-Khinchine representation in terms of a set $\Lambda$ of L\'evy triplets $(b,\Si,\mu)$ satisfying an integrability condition, i.e.~$u(t,x)$ is the solution of
\begin{equation}\label{PDE:Rd}
\partial_t u =\sup_{(b,\Si,\mu)\in \La} A_{b,\Si,\mu} u,\quad\mbox{and}\quad u(0)=f,
\end{equation}
where $A_{b,\Si,\mu}$ is the generator with the L\'evy triplet $(b,\Si,\mu)$, see also Example \ref{genlevy}.
Conversely, starting from the unique solution of \eqref{PDE:Rd}, Hu and Peng \cite{PengHu} give a construction
of the respective nonlinear L\'evy process.
In Nutz and Neufeld \cite{NutzNeuf} the authors consider upper expectations $\EE(\cdot)=\sup_{\P} \E_\P(\cdot)$ over a class of all semimartingales with given differential characteristics in a set $\Lambda$ of L\'evy triplets,
which in \cite{NutzNeuf1} is shown to be analytic. This allows them to construct conditional nonlinear expectations and nonlinear L\'evy processes with general characteristics whose distributions are defined for all measurable functions. Under the conditions
\begin{equation}\label{cond:NN}
  \sup_{(b,\Si,\mu)\in \La} \Big(|b|+|\Si|+\int_{\R^d} |y| \wedge |y|^2\, {\rm d}\mu(y)\Big)<\infty\quad\mbox{and} \quad \lim_{\ep\searrow 0}\sup_{(b,\Si,\mu)\in \La}\int_{|z|\le\varepsilon}|z|^2\, {\rm d}\mu(z)=0
 \end{equation}
it is shown that $u(t,x)=\EE(f(x+X_t))$ is the unique viscosity solution of \eqref{PDE:Rd}. The conditions in \eqref{cond:NN} are weaker than the integrability condition in \cite{PengHu} and allow for instance to consider classes of L\'evy processes with infinite variation jumps. Our main condition \eqref{as:main} in the context of $\R^d$-valued processes is guaranteed under \begin{equation}\label{lev1111}
  \sup_{(b,\Si,\mu)\in \La} \Big(|b|+|\Si|+\int_{\R^d} 1\wedge |y|^2\, {\rm d}\mu(y)\Big)<\infty,
 \end{equation}
which does not exclude any L\'evy triplet at all. In particular, L\'evy processes with non-integrable jumps can be considered, see e.g.~Example \ref{cex2}, and for finite $\Lambda$ the condition \eqref{lev1111} is always satisfied.
In order to obtain uniqueness for the viscosity solution of \eqref{PDE:Rd} one additionally needs the second condition in \eqref{cond:NN} and tightness of the family of L\'evy measures $\{\mu:(b,\Si,\mu)\in\Lambda\}$, which is due to \cite{PengHu}. In Hollender \cite{H2016}
the results of \cite{NutzNeuf} are generalized to upper expectations over state-dependent L\'evy triplets, see also
 K\"uhn \cite{k18} for existence results on the respective integro-differential equations under fairly general conditions.
 A related concept to nonlinear L\'evy processes are second order backward stochastic differential equation with jumps, see Kazi-Tani et al.~\cite{kpz},~\cite{kpz1} and also Soner et al.~\cite{stz}.

The paper is organized as follows. In Section \ref{sec:mainresults} we introduce the notation and discuss our main results which are illustrated
with several examples in Section \ref{sec:ex}. The relation between $\EE$-L\'evy processes and Markovian convolution semigroups is given in Section \ref{sec2}. Finally, in Section \ref{sec3} we prove the main result by constructing
a version of Nisio semigroups on $\BUC(G)$.

\section{Main results}\label{sec:mainresults}

We say that $(\Om,\FF,\EE)$ is a \emph{convex expectation space} if $(\Om,\FF)$ is a measurable space and  $\EE\colon \LL^\infty(\Om,\FF)\to \R$ is a convex expectation which is  continuous from below. As usual $\LL^\infty(\Om,\FF)$ denotes the space of all bounded measurable functions $\Om\to \R$. Recall that a convex expectation on a convex set $M$ with $\R\subset M$ is a functional $\EE\colon M\to\R$ which satisfies
 \begin{enumerate}[]
 \item $\EE(X)\leq \EE(Y)$ whenever $X\leq Y$,
 \item $\EE(\al)=\al$ for all $\al\in \R$, and
 \item $\EE\big(\la X+(1-\la)Y\big)\leq \la \EE(X)+(1-\la)\EE(Y)$ for all $\la\in [0,1]$.
 \end{enumerate}
If in addition $\EE$ is positive homogeneous, i.e.~$\EE(\la X)=\lambda \EE(X)$ for all $\la>0$, then
$\mathcal{E}$ is  called a  \emph{sublinear} expectation and $(\Om,\FF,\EE)$ is a \emph{sublinear} expectation space. A convex expectation is said to be \emph{continuous from below} if $\EE(X_n) \nearrow \EE(X)$ for every increasing sequence $(X_n)$ in $\LL^\infty(\Om,\FF)$ which converges pointwise to $X\in \LL^\infty(\Om,\FF)$.

Let $(\Om,\FF,\EE)$ be a convex expectation space, and let $G$ be an abelian group with a translation invariant metric $d$ such that $(G,d)$ is separable and complete. We denote by $C_b(G)$ and $\BUC(G)$ the spaces of all bounded functions $f\colon G\to\R$ which are continuous and uniformly continuous, respectively. For an $\FF$-$\mathcal B$-measurable random variable $X\colon \Om\to G$ with values in $G$ endowed with the Borel $\sigma$-algebra $\mathcal{B}$, the functional
\[
 \EE\circ X^{-1}\colon C_b(G)\to \R,\quad f\mapsto \EE(f(X))
\]
defines a convex expectation which is called the \emph{distribution} of $X$ under $\EE$.
Given another random variable $Y\colon \Om\to S$ with values in a Polish space $S$, for $f\in C_b(S\times G)$ the function
 \[
  S \to \R, \quad y\mapsto \EE(f(y,X))
 \]
 is bounded and lower semicontinuous. In fact, for $g\in \BUC(S\times G)$, it follows that
 \[
  |\EE(g(y,X))-\EE(g(z,X))|\leq \|g(y,\,\cdot\,)-g(z,\, \cdot \,)\|_\infty
 \]
 for $y,z\in G$ and therefore, $ y\mapsto \EE(g(y,X))$ is uniformly continuous. Approximating $f\in C_b(S\times G)$ from below by a sequence $(g_n)_{n\in \N}\subset \BUC(S\times G)$, see \cite[Remark 5.4 a)]{DKN}, we obtain that $ y\mapsto \EE(f(y,X))$ is lower semicontinuous. Hence, $\EE(f(y,X))|_{y=Y}$ is in $\LL^\infty(\Omega,\FF)$, which shows that
$\EE\big(\EE(f(y,X))|_{y=Y}\big)$ is well-defined.
 Then, $X$ is called \emph{independent} of $Y$ if
 \[
  \EE(f(Y,X))=\EE\big(\EE(f(y,X))|_{y=Y}\big)
 \]
 for all $f\in C_b(S \times G)$.% For an introduction and background to nonlinear expectations we refer to (XXX).

\begin{definition}\label{deflevy}
\begin{enumerate}[a)]
\item We say that $\SS\colon \BUC(G)\to \BUC(G)$ is a convex \emph{kernel} if $(\SS\, \cdot\,)(x)$ is a convex expectation
on $\BUC(G)$ for all $x\in G$. It is called  \emph{continuous from above} if $\SS f_n \searrow \SS f$ (pointwise convergence)
for every decreasing sequence $(f_n)$ in $\BUC(G)$ which converges pointwise to  $f\in \BUC(G)$.

\item A convex kernel $\SS\colon \BUC(G)\to \BUC(G)$ is called a convex \textit{Markovian convolution}
if $\SS f_n \nearrow \SS f$ for every increasing sequence $(f_n)$ in $\BUC(G)$ which converges pointwise to $f\in \BUC(G)$ and
\[
  \big(\SS f\big)(x)=\big(\SS f_x\big)(0)
  \]
for all $f\in \BUC(G)$ and $x\in G$, where $f_x\colon G\to \R$ is given by $y\mapsto f(x+y)$.

\item\label{MCS} We say that $(\SS(t))_{t\geq 0}$ is a convex \textit{Markovian convolution semigroup} on $\BUC(G)$ if
\begin{enumerate}[(i)]
 \item $\SS(t)$ is a convex Markovian convolution for all $t\geq 0$,
 \item $\SS(0)f=f$ for all $f\in \BUC(G)$,
 \item $\SS(s+t)=\SS(s)\SS(t)$ for all $s,t\geq 0$,
 \item $\lim_{t\searrow 0}\|\SS(t)f-f\|_\infty=0$ for all $f\in \BUC(G)$.
 \end{enumerate}
 In this case, we say that $(\SS(t))_{t\geq 0}$ is continuous from above if each $\SS(t)$ is so.
 \item Let $(\Om,\FF,\EE)$ be a convex expectation space. Then, $(X_t)_{t\geq 0}$ is called an \textit{$\EE$-L\'evy process} if
 \begin{enumerate}[(i)]
  \item $X_t\colon \Om\to G$ is measurable for all $t\geq 0$,
  \item $\EE(f(X_0))=f(0)$ for all $f\in C_b(G)$,
  \item $\EE\circ (X_{s+t}-X_s)^{-1}=\EE\circ X_t^{-1}$ for all $s,t\geq 0$,
  \item $X_{s+t}-X_s$ is independent of $(X_{t_1},\ldots,X_{t_n})$ for all $s,t\geq 0$, $n\in \N$, $0\leq t_1< \ldots< t_n\leq s$,
  \item $\EE(f(X_t))\to f(0)$ for all $f\in C_b(G)$, i.e.~$X_t\to X_0$ in distribution as $t\searrow 0$.
 \end{enumerate}
 \end{enumerate}
\end{definition}

\begin{remark}\label{rem:kernel}
Let $\SS:\BUC(G)\to\BUC(G)$ be a convex kernel which is continuous from above. Then, the mapping
\[
\EE_0\colon \BUC(G)\to \R,\quad f\mapsto \big(\SS f\big)(0)
\]
is a convex expectation which is continuous from above. If $\SS$ is in addition a Markovian convolution then, by \cite[Theorem 3.10]{DKN} with $\Omega:=G$, $\FF$ the Borel $\si$-algebra on $G$ and $M=\BUC(G)$, there exists a convex expectation space $(\Omega,\FF,\EE)$ and a random variable $X$ (here the identity on $G$) such that
\[
  \big(\SS f\big)(x)=\EE(f(x+X))
 \]
 for all $f\in C_b(G)$ and $x\in G$.
 By \cite[Remark 5.4 c)]{DKN}
 the convex kernel $\SS$ has a unique extension to a convex kernel $\SS:C_b(G)\to C_b(G)$ which is continuous from above.

  %Moreover, for every $\ep>0$ there exists a compact set $K\subset G$ such that $\EE\big(1_{G\sm K}(X)\big)<\ep$, i.e.~$\EE\circ X^{-1}$ is \emph{tight}.
  %For the details we refer to \cite{DKN}.
\end{remark}

Our first result connects convex Markovian convolution semigroups and $\EE$-L\'evy processes. The proof is given in Section \ref{sec2}.

\begin{theorem}\label{equivlevy}
For every convex Markovian convolution semigroup  $(\SS(t))_{t\geq 0}$ which is continuous from above,
there exists a convex expectation space $(\Om,\FF,\EE)$ and an $\EE$-L\'evy process $(X_t)_{t\geq 0}$ such that
 \begin{equation}\label{defSSEE}
  \big(\SS(t)f\big)(x)=\EE(f(x+X_t))
 \end{equation}
 for every $f\in \BUC(G)$ and $x\in G$.

 Conversely, for every $\EE$-L\'evy process $(X_t)_{t\geq 0}$ on a convex expectation space $(\Om,\FF,\EE)$, the family $(\SS(t))_{t\geq 0}$ defined by \eqref{defSSEE}
is a convex Markovian convolution semigroup.
\end{theorem}

In this paper, we use the following definition of a viscosity solution.
\begin{definition}
 Let $D\subset \BUC(G)$ and $\AA\colon D\subset \BUC(G)\to \BUC(G)$. Then, we say that $u\colon [0,\infty)\to \BUC(G)$ is a \textit{$D$-viscosity subsolution} of the PDE
 \begin{equation}\label{eq:PDE}
  u_t= \AA u
 \end{equation}
 if $u\colon [0,\infty)\to \BUC(G)$ is continuous and for every $t>0$ and $x\in G$ we have
 \[
  \partial_t\psi(t,x)\leq \big(\AA\psi(t)\big)(x)
 \]
for every differentiable $\psi\colon (0,\infty)\to \BUC(G)$ which satisfies $(\psi(t))(x)=(u(t))(x)$, $\psi(s)\geq u(s)$ and $\psi(s)\in D$ for all $s>0$. Here and in the following, we use the notation $\psi(t,x) := (\psi(t))(x)$.

Analogously, $u$ is called a \textit{$D$-viscosity supersolution} of \eqref{eq:PDE}
if $u\colon [0,\infty)\to \BUC(G)$ is continuous and for every $t>0$ and $x\in G$
we have
 \[
  \partial_t\psi(t,x)\geq \big(\AA\psi(t)\big)(x)
 \]
for every differentiable $\psi\colon (0,\infty)\to \BUC(G)$ which satisfies $(\psi(t))(x)=(u(t))(x)$, $\psi(s)\leq u(s)$ and $\psi(s)\in D$ for all $s>0$.

We say that $u$ is a \textit{$D$-viscosity solution} of \eqref{eq:PDE} if $u$ is a $D$-viscosity subsolution and a $D$-viscosity supersolution.
\end{definition}

Now we are ready to state our main result.
Given a family $(A_\la)_{\la\in\La}$ of generators of L\'evy processes
we provide the existence of an $\EE$-L\'evy process on a sublinear expectation space with generator $\sup_{\la\in\La} A_\la$. The corresponding sublinear Markovian convolution semigroup
is a viscosity solution of the fully nonlinear PDE $u_t=\sup_{\la\in \La} A_\la u$ with $u(0,\cdot)=f\in \BUC(G)$.
The proof is postponed to Section \ref{sec3}.

\begin{theorem}\label{main}
  Let $\La\neq \es$ be an index set. Assume that the following holds:
\begin{enumerate}
 \item[(A1)] For each $\la\in \La$ let $A_\la\colon D(A_\la)\subset \BUC(G)\to \BUC(G)$ be the generator
 of a Markovian convolution semigroup $(S_\la(t))_{t\geq 0}$ of linear operators.
 \item[(A2)] The subspace
 \[
 D:=\bigg\{ f\in \bigcap_{\la\in \La}D(A_\la) : \big\{A_\la f\colon \la\in \La\big\}\;\text{is bounded and uniformly equicontinuous}\bigg\}
 \]
 is dense in $\BUC(G)$.
 \end{enumerate}

 Then, there exists a sublinear expectation space $(\Om,\FF,\EE)$ and an $\EE$-L\'evy process $(X_t)_{t\geq 0}$ such that for each $f\in \BUC(G)$ the function
 \begin{equation}\label{function:u}
  \left(u(t)\right)(x):=\EE(f(x+X_t)),\quad  t\geq 0,\, x\in G
 \end{equation}
 is a $D$-viscosity solution of the fully nonlinear PDE
 \begin{eqnarray}
  u_t(t,x)&=&\sup_{\la\in \La} \big(A_\la u(t)\big)(x), \quad (t,x)\in (0,\infty)\times G,\label{PDE1}\\
  u(0,x)&=& f(x),\quad x\in G.\label{PDE2}
 \end{eqnarray}
 Moreover, there exists a set $\mathcal{P}$ of probability measures on $(\Omega,\FF)$ such that
$\EE(Y)=\sup_{\P\in \PP}\E_\P(Y)$ for all $Y\in\LL^\infty(\Om,\FF)$.
\end{theorem}

 \begin{remark}\label{rem:Picard}
  In the situation of Theorem \ref{main}, if each $A_\la\colon \BUC(G)\to \BUC(G)$ is a bounded linear operator and
  $\sup_{\la\in \La} \|A_\la\|<\infty$,  then the mapping
  \[
  \BUC(G)\to\BUC(G),\quad u\mapsto \sup_{\la\in \La} A_\la u
  \]
is Lipschitz continuous. Therefore, by the Picard-Lindel\"of theorem, the function $u$ in \eqref{function:u} is a classical solution of the fully nonlinear PDE \eqref{PDE1}-\eqref{PDE2}, which satisfies $u\in C^1([0,\infty);\BUC(G))$.
 \end{remark}

 In most applications, the conditions (A1) and (A2) in Theorem \ref{main} can be easily verified as shown in
 Section \ref{sec:ex}. For the sake of illustration, we consider the
 case $G=\R^d$, where the L\'evy-Khintchine formula characterizes generators of Markovian convolution semigroups of linear operators by means of so-called L\'evy triplets, see e.g. Applebaum~\cite{MR2512800} or Sato~\cite{MR1739520}. Given a set $\La$ of L\'evy triplets $(b,\Si,\mu)$, i.e.~$b\in \R^d$, $\Si\in \R^{d\times d}$ is a symmetric positive semidefinite matrix and $\mu$ is a L\'evy measure, the condition \eqref{lev1111} is sufficient to guarantee (A2) with $\BUC^2(\R^d)\subset D$. Here, $\BUC^2(\R^d)$ denotes the space of all functions which are twice differentiable with bounded uniformly continuous derivatives up to order 2.  For more details, we refer to Example \ref{genlevy} which contains $G=\R^d$ as a special case.

\begin{remark} Let $\psi\in C_b^{2,3}((0,\infty)\times \R^d)$, where $C_b^{2,3}((0,\infty)\times \R^d)$ stands for the space of all functions of $(t,x)\in (0,\infty)\times\R^d$ for which all partial derivatives up to order 2 in $t$ and up to order 3 in $x$ exist, are continuous and bounded.
Then,
\[
   \lim_{h\to 0}\sup_{x\in \R^d}\bigg|\frac{\psi(t+h,x)-\psi(t,x)}{h}-\partial_t\psi(t,x)\bigg| = 0
\]
for all $t>0$ and therefore, $\psi\colon (0,\infty)\to \BUC(\R^d)$ is differentiable with $\psi(s)\in \BUC^2(\R^d)$ for all $s>0$ using the identification $(\psi(s))(x):=\psi(s,x)$. Therefore, the class of test functions considered in the framework of $\BUC^2(\R^d)$-viscosity solutions includes the class $C_b^{2,3}((0,\infty)\times \R^d)$ of test functions, which is often considered in classical viscosity theory, see e.g. Denis et al.~\cite{MR2754968} or Hu and Peng~\cite{PengHu}. Assuming in addition to \eqref{lev1111} that for every $\ep>0$ there exists some $M>0$ such that
 \begin{equation}\label{lev2}
  \sup_{(b,\Si,\mu)\in \La} \mu\big(\R^d\sm B(0,M)\big)<\ep ,
 \end{equation}
where $B(0,M):=\{x\in\R^d:|x|\le M\}$, and the second condition in \eqref{cond:NN}
one obtains from  \cite[Corollary 53]{PengHu} the uniqueness of the viscosity solution of the PDE \eqref{PDE:Rd}.

%Note that the conditions (1.3) and (1.4) in Nutz and Neufeld \cite{NutzNeuf} imply \eqref{lev1111} - \eqref{lev3}. However, while conditions (1.3) and (1.4) in \cite{NutzNeuf} exclude all L\'evy triplets with non-integrable jumps, \eqref{lev1111} does not exclude any L\'evy triplet at all. In particular, for finite $\La$ the condition \eqref{lev1111} is always satisfied.
%Below, we will give examples of $\EE$-L\'evy processes which satisfy \eqref{lev1111}
%but not \eqref{lev3} in Example \ref{cex1}, and not (1.3) in \cite{NutzNeuf} in Example \ref{cex2}, respectively.
\end{remark}

Let $C_0(G)$ be the closure of the space $C_c(G)$ of all continuous functions with compact support w.r.t.~the supremum norm $\|\cdot\|_\infty$. Note that the existence of a function in $C_0(G)\setminus \{0\}$ already implies that $G$ is locally compact and vice versa since $G$ is a topological abelian group, which is metrizable. The following additional condition on the generators $(A_\lambda)_{\lambda\in \Lambda}$ implies that the related Markovian convolution semigroup is
continuous from above.
\begin{proposition}\label{prop:A3}
In addition to the assumptions in Theorem \ref{main}, suppose:
\begin{enumerate}
 \item[(A3)] For every $\varepsilon >0$ there exists $\varphi\in \bigcap_{\la\in \La} \big(D(A_\la)\cap C_0(G)\big)$ with $0\leq \varphi\leq 1$, $\varphi(0)=1$ and $\sup_{\la\in \La}\|A_\la\varphi\|_\infty\leq \varepsilon$.
\end{enumerate}
 Then, the related Markovian convolution semigroup $(\SS(t)f)(x)=\EE(f(x+X_t))$, $t\ge 0$, is
 continuous from above on $\BUC(G)$.
 %$\EE\circ X^{-1}_t$ is tight for all $t\geq 0$.
\end{proposition}
The proof is given at the end of Section~5. In line with Remark \ref{rem:kernel},
under (A3) the Markovian convolution semigroup $(\SS(t))_{t\ge 0}$ has a unique extension from $\BUC(G)$ to $C_b(G)$ which is continuous from above. Moreover, continuity from above implies a dual max-representation for the sublinear expectation $\mathcal{E}$ in terms of probability measures, see \cite[Lemma 2.4, Lemma 3.2]{DKN}.
For instance, in the case $G=\R^d$, where the generators are given by L\'evy triplets $(b,\Si,\mu)\in\Lambda$, the condition (A3) holds if \eqref{lev1111} is satisfied and the set of L\'evy measures $\{\mu : (b,\Si,\mu)\in \La\}$ is tight. %Note that $C_0(G)\neq \{0\}$ if and only if $G$ is locally compact.
%Therefore, (A3) can only be satisfied if $G$ is locally compact.

\section{Examples}\label{sec:ex}

Let $\ca_+^1(G)$ be the set of all Borel probability measures on $G$.

\begin{example}[Compound Poisson processes]\label{compois}
 For $\la \geq 0$ and $\mu\in \ca_+^1(G)$, let
 \[
  \big(A_{\la,\mu} f\big)(x):=\la \int_G f(x+y)-f(x)\, {\rm d}\mu(y),\quad f\in \BUC(G),\,x\in G.
 \]
 Then, $A_{\la,\mu}\colon \BUC(G)\to \BUC(G)$ is a bounded linear operator which satisfies the positive maximum principle (cf.~\cite[Definition 4.5.1]{MR1873235}), i.e.~for $f\in \BUC(G)$ and $x_0\in G$ with $f(x_0)=\max_{x\in G} f(x)\geq 0$ one has $\big(A_{\la,\mu} f\big)(x_0)\leq 0$. Further, since $A_{\lambda,\mu}$ is bounded and linear, it generates the linear uniformly continuous semigroup $(e^{tA_{\la,\mu}})_{t\geq 0}$. Recall that for a bounded linear operator $B\colon \BUC(G)\to \BUC(G)$ the exponential $e^{B}:=\sum_{k=0}^\infty \tfrac{1}{k!}B^k$ of $B$ is again a bounded linear operator $\BUC(G)\to \BUC(G)$. We first show that $S_{\la,\mu} (t):=e^{tA_{\la,\mu}}$ satisfies
\[
  \big(S_{\la,\mu}(t)f\big)(x)=\E(f(x+J_t))=\int_G f(x+y)\, {\rm d}\big(\P\circ J_t^{-1}\big)(y),\quad f\in \BUC(G)
 \]
for all $t\geq 0$, where $(J_t)_{t\geq 0}$ is a compound Poisson process with rate $\lambda$ and jump size distribution $\mu$
on a probability space $(\Om,\FF,\P)$. In particular, $\big(S_{\la,\mu} (t)\big)_{t\geq 0}$ is a linear Markovian convolution semigroup. Indeed, let $J_t=\sum_{i=1}^{N_t} Y_i$ for an i.i.d.~sequence $(Y_i)_{i\in \N}$ of random variables $Y_i\colon \Om\to G$ such that $\P\circ Y_i^{-1} =  \mu$, and a random variable $N_t\colon \Om\to \N_0$ which is independent of $(Y_i)_{i\in \N}$ and satisfies $\P(N_t=n)=e^{-\la t}\frac{(\la t)^n}{n!}$ for all $n\in \N$. Then, for $f\in \BUC(G)$ and $x\in G$ we have
\begin{align*}
  \E(f(x+J_t))&=\sum_{n=0}^\infty \E(f(x+Y_1+\ldots +Y_n)) e^{-\la t}\frac{(\la t)^n}{n!}\\
  &=e^{-\la t}\sum_{n=0}^\infty \frac{t^n\la^n}{n!}\int_G\cdots \int_G f(x+y_1+\ldots +y_n)\, {\rm d}\mu(y_1)\cdots {\rm d}\mu (y_n)\\
  &= e^{-\la t}\sum_{n=0}^\infty \frac{\big(t^n(A_{\la,\mu}+\la )^nf\big)(x)}{n!}=e^{-\la t}\big(e^{t(A_{\la,\mu}+\la )}f \big)(x)\\
  &=\big(e^{tA_{\la,\mu}}f\big)(x)=\big(S_{\la,\mu}(t)f\big)(x),
 \end{align*}
where we used $\big((A_{\la,\mu}+\la)f\big)(x)=\la \int_Gf(x+y)\, {\rm d}\mu(y)$.

 Now, assume that $\La\subset [0,\infty)$ is bounded and $\QQ\subset \ca_+^1(G)$. Since
 \[\left\{ f\in \BUC(G)\, : \, \begin{array}{l} \{A_{\la,\mu}f\colon (\la,\mu)\in \La\times \QQ\}\;\text{is bounded } \\ \text{and uniformly equicontinuous}\end{array}\right\}\\=\BUC(G),\]
the assumptions (A1) and (A2) are satisfied with $D=\BUC(G)$. Hence, by Theorem \ref{main} and Remark \ref{rem:Picard}
there exists a nonlinear expectation space $(\Om,\FF,\EE)$ and an $\EE$-L\'evy process $(X_t)_{t\geq 0}$ such that for all $f\in \BUC(G)$ the function
  \[
  u(t,x):=\big(u(t)\big)(x):=\EE(f(x+X_t)),\quad t\geq 0,\,x\in G,
 \]
 is the unique classical solution of the fully nonlinear PIDE
 \begin{eqnarray*}
  u_t(t,x)&=&\sup_{(\la,\mu)\in \La\times \QQ} \big(A_{\la,\mu} u(t)\big)(x), \quad  (t,x)\in (0,\infty)\times G,\\
  u(0,x)&=&f(x),\quad x\in G,
 \end{eqnarray*}
with $u\in C^1([0,\infty);\BUC(G))$.
\end{example}

\begin{example}[L\'evy processes on $\R^d$]\label{genlevyRd}
 Let $G=\R^d$ with $d\in \N$ and consider a set $\La$ of L\'evy triplets $(b,\Si,\mu)$, i.e. $b\in \R^d$, $\Si\in \R^{d\times d}$ is a symmetric positive semidefinite matrix and $\mu$ is a L\'evy measure on $\R^d$, i.e. a Borel measure satisfying
 \[
  \mu(\{0\})=0\quad \text{and}\quad \int_{\R^d} 1\wedge |y|^2\, {\rm d}\mu(y)<\infty.
 \]
 By the L\'evy-Khintchine formula, for each L\'evy triplet $(b,\Si,\mu)$ there exists a Markovian convolution semigroup $(S_{b,\Si,\mu}(t))_{t\geq 0}$ of linear operators on $\BUC(\R^d)$ with generator $A_{b,\Si,\mu}\colon D(A_{b,\Si,\mu})\subset \BUC(\R^d)\to \BUC(\R^d)$. Moreover, $\BUC^2(\R^d)\subset D(A_{b,\Si,\mu})$, where $\BUC^2(\R^d)$ denotes the space of all functions $\R^d\to \R$ which are twice differentiable with bounded uniformly continuous derivatives up to order 2.  For $f\in \BUC^2(\R^d)$, the function $A_{b,\Si,\mu}f$ is given by
 \[
  \big(A_{b,\Si,\mu}f\big)(x)= \langle b, \nabla f(x)\rangle + \frac{1}{2}{\rm tr} \big(\Si \nabla^2f(x)\big)+\int_{\R^d} f(x+y)-f(x)-\langle \nabla f(x), h(y)\rangle\, {\rm d}\mu(y)
 \]
 for $x\in \R^d$. Here, the function $h\colon \R^d\to \R^d$ is defined by $h(y)=y$ for $|y|\leq 1$, and $h(y)=0$ whenever $|y| > 1$. In particular, (A1) is satisfied. Assume that
%Then, by the L\'evy-Khintchine formula (see e.g. \cite[Theorem 5.7.3]{MR874529}), for each L\'evy triplet $(b,\Si,\mu)$ the operator $A_{b,\Si,\mu}$ generates a Markovian convolution semigroup $(S_{b,\Si,\mu}(t))_{t\geq 0}$ of linear operators on $\BUC(H)$. In particular, (A1) is satisfied. Assume that
% For any L\'evy triplet $(b,\Si,\mu)$ consider
 %\[
 % \big(A_{b,\Si,\mu}f\big)(x):= \langle b, \nabla f(x)\rangle + \frac{1}{2}{\rm tr} \big(\Si \nabla^2f(x)\big)+\int_{H} f(x+y)-f(x)-\langle \nabla f(x), h(y)\rangle\, {\rm d}\mu(y)
 %\]
 %for $x\in H$ and $f\in D(A_{b,\Si,\mu}):=\big\{g\in \BUC(H)\colon A_{b,\Si,\mu}g\in \BUC(H)\big\}$, where the function $h\colon H\to H$ is defined by $h(y)=y$ for $\|y\|\leq 1$, and $h(y)=0$ whenever $\|y\| > 1$.
%Then, by the L\'evy-Khintchine formula (see e.g. \cite[Theorem 5.7.3]{MR874529}), for each L\'evy triplet $(b,\Si,\mu)$ the operator $A_{b,\Si,\mu}$ generates a Markovian convolution semigroup $(S_{b,\Si,\mu}(t))_{t\geq 0}$ of linear operators on $\BUC(H)$. In particular, (A1) is satisfied. Assume that
 \begin{equation}\label{lev1}
  C:=\sup_{(b,\Si,\mu)\in \La} \bigg(|b|+\|\Si\|_{{\rm tr}}+\int_{\R^d} 1\wedge |y|^2\, {\rm d}\mu(y)\bigg)<\infty,
 \end{equation}
 where $\|\cdot \|_{{\rm tr}}$ denotes the trace norm. We will verify that (A2) is satisfied under \eqref{lev1}. Let
 \[
  D:=\bigg\{f\in \bigcap_{(b,\Si,\mu)\in \La} D(A_{b,\Si,\mu})\, : \, \begin{array}{l} \{A_{b,\Si,\mu}f\, : \, (b,\Si,\mu)\in \La\} \text{ is bounded } \\ \text{ and uniformly equicontinuous}\end{array} \bigg\}.
 \]
%We denote by  $\BUC^2(H)$ the space of all functions which are twice differentiable with bounded uniformly continuous derivatives up to order 2.
Since $\BUC^2(\R^d)$  is dense in $\BUC(\R^d)$, it suffices to show that $\BUC^2(\R^d)\subset D$. Let $f\in \BUC^2(\R^d)$. Then, $f\in D(A_{b,\Si,\mu})$ for any L\'evy triplet $(b,\Si,\mu)\in \La$. In the sequel, we denote by $\|\nabla^2 f(x)\|$ the operator norm of the matrix $\nabla^2 f(x)\colon \R^d\to \R^d$ for all $x\in \R^d$ and by $\|\nabla^2f\|_\infty:=\sup_{x\in \R^d}\|\nabla^2 f(x)\|$. Then, by Taylor's theorem we have
 \begin{align*}
 \big|\big(&A_{b,\Si,\mu}f\big)(x)\big|\leq |b|  |\nabla f(x)|+\tfrac 12\|\Si\|_{{\rm tr}}  \|\nabla^2f(x)\| \\
 & \qquad +\int_{\R^d}\big|f(x+y)-f(x)-\langle \nabla f(x), h(y)\rangle \big|\, {\rm d}\mu(y)\\
   &\leq |b| \|\nabla f\|_\infty+\tfrac 12\|\Si\|_{{\rm tr}}  \|\nabla^2 f\|_\infty+ \max \bigg\{2\| f\|_\infty, \frac{1}{2}\|\nabla^2f\|_\infty\bigg\} \int_{\R^d} 1\wedge |y|^2\, {\rm d}\mu(y)\\
   &\leq 2C \max\big\{\|f\|_\infty,\|\nabla f\|_\infty, \|\nabla^2 f\|_\infty\big\}
 \end{align*}
 for all $x\in \R^d$ and all $(b,\Si,\mu)\in \La$, so that
 \begin{equation}\label{levy1}
  \sup_{(b,\Si,\mu)\in \La} \|A_{b,\Si,\mu}f\|_\infty\leq 2C\max \big\{\|f\|_\infty,\|\nabla f\|_\infty, \|\nabla^2 f\|_\infty\big\}.
 \end{equation}
 Let $\ep>0$. Since $f\in \BUC^2(\R^d)$ there exists $\de>0$ such that for all $x,z\in \R^d$ with $|x-z|\leq \de$ and $\theta \in \{f,\nabla f, \nabla^2 f\}$ it holds
 \begin{align*}
 & 2C\sup_{y\in \R^d }|\theta (x+y)-\theta (z+y)|\leq \ep. %\left(|f(x+y)-f(z+y)|\vee |\nabla f(x+y)-\nabla f(z+y)|\vee |\nabla^2 f(x+y)-\nabla^2 f(z+y)|\right)\\ & \leq \ep.
 \end{align*}
 Let $x,z\in \R^d$ be fixed with $|x-z|\leq \de$ and $g\colon \R^d\to \R$ be defined by
 \[
  g(y):=f(x+y)-f(z+y)
 \]
 for all $y\in \R^d$, so that $g\in \BUC^2(\R^d)$ with $2C\max \big\{\|g\|_\infty,\|\nabla g\|_\infty, \|\nabla^2g\|_\infty\big\}\leq \ep$. By \eqref{levy1} we get
 \begin{align*}
  \big|\big(A_{b,\Si,\mu}f\big)(x)-\big(A_{b,\Si,\mu}f\big)(z)\big|&= \big|\big(A_{b,\Si,\mu} g\big)(0)\big|\leq 2C \max \big\{\|g\|_\infty,\|\nabla g\|_\infty, \|\nabla^2g\|_\infty\big\}\leq \ep
 \end{align*}
 for all $(b,\Si,\mu)\in \La$, which shows that (A2) is satisfied. Therefore, by Theorem \ref{main}
 there exists a sublinear expectation space $(\Om,\FF,\EE)$ and an $\EE$-L\'evy process $(X_t)_{t\geq 0}$ such that for all $f\in \BUC(\R^d)$ the function
 \[
  u(t,x):=\big(u(t)\big)(x):=\EE(f(x+X_t)),\quad t\geq 0,\, x\in \R^d,
 \]
 is a $\BUC^2(\R^d)$-viscosity solution of the fully nonlinear PIDE
 \begin{eqnarray*}
  u_t(t,x)&=&\sup_{(b,\Si,\mu)\in \La} \big(A_{b,\Si,\mu} u(t)\big)(x), \quad (t,x)\in (0,\infty)\times \R^d,\\
  u(0,x)&=&f(x),\quad x\in \R^d.
 \end{eqnarray*}
%Note that \eqref{lev1} is fulfilled whenever $\La$ is finite.
\end{example}

\begin{example}[L\'evy processes on real separable Hilbert spaces]\label{genlevy}
 Let $G=H$ be a real separable Hilbert space. Similar to the previous example, a L\'evy triplet $(b, \Si,\mu)$ consists of a vector $b\in H$, a symmetric positive semidefinite trace class operator $\Si\colon H\to H$ and a L\'evy measure $\mu$ on $H$, i.e. a Borel measure satisfying
 \[
  \mu(\{0\})=0\quad \text{and}\quad \int_H 1\wedge \|y\|^2\, {\rm d}\mu(y)<\infty.
 \]
 By the L\'evy-Khintchine formula (see e.g. \cite[Theorem 5.7.3]{MR874529}), for each L\'evy triplet $(b, \Si,\mu)$ there exists a Markovian convolution semigroup $(S_{b,\la \Si,\mu}(t))_{t\geq 0}$ of linear operators on $\BUC(H)$ with generator $A_{b,\la \Si,\mu}\colon D(A_{b,\la\Si,\mu})\subset \BUC(H)\to \BUC(H)$. In contrast to the previous example the space $\BUC^2(H)$ of all twice differentiable function $H\to \R$ with bounded uniformly continuous derivatives up to order 2 is no longer dense in $\BUC(H)$ (cf. \cite[Section 2.2]{MR874529}). In order to guarantee that $D$ is dense in $\BUC(H)$ we therefore have to keep the convariance operator fixed and also restrict the set of L\'evy measures. We therefore consider a fixed symmetric positive semidefinite trace class operator $Q\colon H\to H$ and denote by $\De_Q:=A_{(0,Q,0)}$ the generator of the heat semigroup with covariance operator $Q$. Then, $D_Q:=D(\De_Q)\cap \BUC^1(H)$ is dense in $\BUC(H)$ (see e.g. \cite[Theorem 3.5.3]{MR1985790}) and for $f\in D_Q$, the function $A_{b,Q,\mu}f$ is given by
 \[
  \big(A_{b,Q,\mu}f\big)(x)= \langle b, \nabla f(x)\rangle + \De_Q f(x)+\int_{H} f(x+y)-f(x)-\langle \nabla f(x), h(y)\rangle\, {\rm d}\mu(y)
 \]
 for $x\in H$. Here, the function $h\colon H\to H$ is defined by $h(y)=y$ for $\|y\|\leq 1$, and $h(y)=0$ whenever $\|y\| > 1$. Now, let $\La$ be a set of triplets $(b,\la,\mu)$, where $b\in H$, $\la\geq 0$ and $\mu$ is a L\'evy measure, satisfying
 \begin{equation}\label{levH1}
  C:=\sup_{(b,\la,\mu)\in \La} \bigg(\|b\|+\la +\int_{H} 1\wedge \|y\|\, {\rm d}\mu(y)\bigg)<\infty,
 \end{equation}
 where $\|\cdot \|_{{\rm tr}}$ denotes the trace norm. Then, (A1) is satisfied and, in the sequel, we will verify that (A2) is satisfied under \eqref{levH1}. Since $D_\Q$ is dense in $\BUC(H)$, it suffices to show that $D_Q\subset D$. Let $f\in D_Q$. Then, $f\in D(A_{b,\Si,\mu})$ for any L\'evy triplet $(b,\Si,\mu)\in \La$. We denote by $\|\nabla^2 f(x)\|$ the operator norm of the bounded linear operator $\nabla^2 f(x)\colon H\to H$ for all $x\in H$ and by $\|\nabla^2f\|_\infty:=\sup_{x\in H}\|\nabla^2 f(x)\|$. Then, by Taylor's theorem we have
 \begin{align*}
 \big|\big(&A_{b,\la Q,\mu}f\big)(x)\big|\leq \|b\|  \|\nabla f(x)\|+\la \|\De_Q f(x)\| \\
 & \qquad +\int_{H}\big|f(x+y)-f(x)-\langle \nabla f(x), h(y)\rangle \big|\, {\rm d}\mu(y)\\
   &\leq \|b\| \|\nabla f\|_\infty + \la \|\De_Q f(x)\|+ \bigg(2\| f\|_\infty+\|\nabla f\|_\infty\bigg) \int_{H} 1\wedge \|y\|\, {\rm d}\mu(y)\\
   &\leq 2C \max\big\{\|f\|_\infty + \|\nabla f\|_\infty, \|\De_Q f\|_\infty\big\}
 \end{align*}
 for all $x\in H$ and all $(b,\la,\mu)\in \La$, so that
 \begin{equation}\label{levyH1}
  \sup_{(b,\la,\mu)\in \La} \|A_{b,\la \Si,\mu}f\|_\infty\leq 2C\max \big\{\|f\|_\infty+\|\nabla f\|_\infty, \|\De_Q f\|_\infty\big\}.
 \end{equation}
 Let $\ep>0$. Since $f\in \BUC^2(H)$ there exists $\de>0$ such that for all $x,z\in H$ with $\|x-z\|\leq \de$ and $\theta \in \{f,\nabla f, \De_Q f\}$ it holds
 \begin{align*}
 & 4C\sup_{y\in H }\|\theta (x+y)-\theta (z+y)\|\leq \ep. %\left(|f(x+y)-f(z+y)|\vee |\nabla f(x+y)-\nabla f(z+y)|\vee |\nabla^2 f(x+y)-\nabla^2 f(z+y)|\right)\\ & \leq \ep.
 \end{align*}
 Let $x,z\in H$ be fixed with $\|x-z\|\leq \de$ and $g\colon H\to \R$ be defined by
 \[
  g(y):=f(x+y)-f(z+y)
 \]
 for all $y\in H$, so that $g\in \De_Q$ with $2C\max \big\{\|g\|_\infty+\|\nabla g\|_\infty, \|\De_Q g\|_\infty\big\}\leq \ep$. By \eqref{levyH1} we get
 \begin{align*}
  \big|\big(A_{b,\Si,\mu}f\big)(x)-\big(A_{b,\Si,\mu}f\big)(z)\big|&= \big|\big(A_{b,\Si,\mu} g\big)(0)\big|\leq 2C \max \big\{\|g\|_\infty+\|\nabla g\|_\infty, \|\De_Q g\|_\infty\big\}\\
  &\leq \ep
 \end{align*}
 for all $(b,\la ,\mu)\in \La$, which shows that (A2) is satisfied. Therefore, by Theorem \ref{main}
 there exists a sublinear expectation space $(\Om,\FF,\EE)$ and an $\EE$-L\'evy process $(X_t)_{t\geq 0}$ such that for all $f\in \BUC(H)$ the function
 \[
  u(t,x):=\big(u(t)\big)(x):=\EE(f(x+X_t)),\quad t\geq 0,\, x\in H,
 \]
 is a $\BUC^2(H)$-viscosity solution of the fully nonlinear PIDE
 \begin{eqnarray*}
  u_t(t,x)&=&\sup_{(b,\Si,\mu)\in \La} \big(A_{b,\Si,\mu} u(t)\big)(x), \quad (t,x)\in (0,\infty)\times H,\\
  u(0,x)&=&f(x),\quad x\in H.
 \end{eqnarray*}
\end{example}

\begin{example}[L\'evy processes on the $d$-dimensional Torus $\T^d$]\label{genlevtorus}
 Let $G=\T^d$, where $\T^d$ denotes the $d$-dimensional Torus represented by $(-\pi,\pi]^d$. We say that $(b,\Si,\mu,\nu)$ is a \textit{L\'evy quadruple} if $b\in \R^d$, $\Si\in \R^{d\times d}$ is a symmetric positive semidefinite matrix, $\mu$ is a positive finite measure on $\T^d$ and $\nu$ is a positive measure on $\T^d$ with $\nu(\{0\})=0$ and
 \[
  \int_{\T^d}|y|^2\, {\rm d}\nu(y)<\infty.
 \]
 This definition is motivated by the L\'evy-It\^{o} decomposition of a L\'evy process in $\R^d$, where the L\'evy measure is further decomposed into a measure describing the large jumps (here: $\mu$) and another measure describing the small jumps (here: $\nu$). For each L\'evy quadruple $(b,\Si,\mu,\nu)$ we consider
  \begin{align*}
  \big(A_{b,\Si,\mu,\nu}f\big)(x)&:=b\cdot \nabla f(x)+ \frac{1}{2}{\rm tr} \big(\Si \nabla^2f(x)\big)+\int_{\T^d} f(x+y)-f(x)\, {\rm d}\mu(y)\\
  &\quad\; +\int_{\T^d} f(x+y)-f(x)-\nabla f(x) \cdot y\, {\rm d}\nu(y)
 \end{align*}
 for $x\in \T^d$ and $f\in D(A_{b,\Si,\mu,\nu}):=\big\{g\in C(\T^d)\, : \, A_{b,\Si,\mu,\nu} g\in C(\T^d)\big\}$.
Recall that $C(\T^d)$ can be identified with the set of all $2\pi$-periodic continuous functions on $\R^d$.

 We start by proving that for any L\'evy quadruple $(b,\Si,\mu,\nu)$ the operator $A_{b,\Si,\mu,\nu}$ generates a Markovian convolution semigroup $(S_{b,\Si,\mu,\nu}(t))_{t\geq 0}$ of linear operators on $\T^d$. In order to do so, let $(b,\Si,\mu,\nu)$ be a L\'evy quadruple and define
 \[
  \eta f:=\int_{(-\pi,\pi]^d} f(y)\, {\rm d}\nu(y)+\sum_{k\in \Z^d\sm \{0\}} \la_k \int_{(-\pi,\pi]^d} f(x+2\pi k)\, {\rm d}\mu(y)
 \]
for $f\in \LL^\infty(\R^d)$, where $\la_k\geq 0$ for all $k\in \Z^d\sm\{0\}$ and $\sum_{k\in \Z^d\sm\{0\}}\la_k=1$. Then, $\eta$ is a L\'evy measure on $\R^d$, see e.g.~Example \ref{genlevy} with $H=\R^d$, so that $(b,\Si,\eta)$ is a L\'evy triplet. Hence, there exists a Markovian convolution semigroup $(S_{b,\Si,\eta}(t))_{t\geq 0}$ of linear operators on $\R^d$ with generator $A_{b,\Si,\eta}$. As the space $C(\T^d)$ of all $2\pi$-periodic continuous functions is a closed subspace of $\BUC(\R^d)$, which is invariant under $S_{b,\Si,\eta}(t)$ for all $t\geq 0$, we obtain that
\[
 S(t):=\big(S_{b,\Si,\eta}(t)\big)|_{C(\T^d)}, \quad t\geq 0
\]
defines a Markovian convolution semigroup of linear operators on $C(\T^d)$. Let $A$ denote the generator of the semigroup $(S(t))_{t\geq 0}$. As $C(\T^d)$ is a closed subspace of $\BUC(\R^d)$ and $\sum_{k\in \Z^d\sm \{0\}}\la_k=1$, we get that
\[
 Af=A_{b,\Si,\eta}f=A_{b,\Si,\mu,\nu}f
\]
for all $f\in D(A)$. In particular, $A_{b,\Si,\mu,\nu}$ is the generator of $(S(t))_{t\geq 0}$.\\

Now, let $\La$ be a set of L\'evy quadruples with
 \begin{equation}\label{toruslev}
  \sup_{(b,\Si,\mu,\nu)\in \La} \bigg(|b|+|\Si|+\mu(\T^d)+\int_{\T^d}|y|^2\, {\rm d}\nu(y)\bigg)<\infty.
 \end{equation}
 Then, in a similar way as in Example \ref{genlevy}, one can show that for every $f\in C^2(\T^d)=\BUC^2(\T^d)$, the set $\{A_{b,\Si,\mu,\nu}f\colon (b,\Si,\mu,\nu)\in \La\}$ is bounded and equicontinuous. Therefore, the assumptions (A1) and (A2) are satisfied. Hence, there exists a nonlinear expectation space $(\Om,\FF,\EE)$ and an $\EE$-L\'evy process $(X_t)_{t\geq 0}$ such that for all $f\in C(\T^d)$ the function
 \[
  u(t,x):=\big(u(t)\big)(x):=\EE(f(x+X_t)), \quad t\geq 0, x\in \T^d,
 \]
 is a $C^2(\T^d)$-viscosity solution of the fully nonlinear PIDE
 \begin{eqnarray*}
  u_t(t,x)&=&\sup_{(b,\Si,\mu,\nu)\in \La} \big(A_{b,\Si,\mu,\nu} u(t)\big)(x), \quad (t,x)\in (0,\infty)\times \T^d,\\
  u(0,x)&=&f(x),\quad x\in \T^d.
 \end{eqnarray*}
We also refer to Hunt \cite{MR0079232} for a L\'evy-Khintchine formula on compact Lie groups.
\end{example}

\begin{example}
 Let $A\colon D(A)\subset \BUC(G)\to \BUC(G)$ be the generator of a Markovian convolution semigroup $(S(t))_{t\geq 0}$ and $\La\subset [0,\infty)$ bounded. For all $\la\in \La$ let $A_\la:=\la A$ and $S_\la(t):=S(\la t)$ for $t\geq 0$. Then, the assumptions (A1) and (A2) are satisfied with $D=D(A)$. Hence, there exists a nonlinear expectation space $(\Om,\FF,\EE)$ and an $\EE$-L\'evy process $(X_t)_{t\geq 0}$ such that for all $u_0\in \BUC(G)$ the function
 \[
  u(t,x):=\EE(f(x+X_t)),\quad t\geq 0, x\in G
 \]
 is a $D(A)$-viscosity solution of the fully nonlinear PDE
 \begin{eqnarray*}
  u_t(t,x)&=&\sup_{\la\in \La} \big(\la A u(t)\big)(x), \quad (t,x)\in (0,\infty)\times G,\\
  u(0,x)&=&f(x),\quad x\in G.
 \end{eqnarray*}
 For instance, if $A$ is the generator of a cylindrical Wiener process on a separable Hilbert space, one obtains
an $\EE$-L\'evy process which can be viewed as a cylindrical $G$-Wiener process.
\end{example}

\begin{example}\label{cex1}%[Approximating Brownian Motion by compound Poisson processes]
 For all $h>0$ let $\mu_h:=\frac{1}{h^2}\de_h$ and consider $\La:=\big\{(0,0,\mu_h)\, :\, h>0\big\}$ in Example \ref{genlevy} with $d=1$. Then, we have that
 \[
  \sup_{h>0} \int_\R 1\wedge |y|^2\, {\rm d}\mu_h(y)=\sup_{h>0} \frac{1}{h^2} \int_\R |y|^2\, {\rm d}\de_h(y)=1,
 \]
so that the assumptions (A1) and (A2) are satisfied. However, the second condition in \eqref{cond:NN} does not hold. By Example~\ref{genlevy}, there exists a sublinear expectation space $(\Om,\FF,\EE)$ and an $\EE$-L\'evy process $(X_t)_{t\geq 0}$ such that for all $f\in \BUC(\R)$ the function
 \[
  u(t,x):=\EE(f(x+X_t)),\quad t\geq 0, x\in \R,
 \]
 is a $\BUC^2(\R)$-viscosity solution of the fully nonlinear PIDE
 \begin{eqnarray*}
  u_t(t,x)&=&\sup_{h>0} \big(A_{0,0,\mu_h} u(t)\big)(x), \quad (t,x)\in (0,\infty)\times \R,\\
  u(0,x)&=&f(x),\quad x\in \R.
 \end{eqnarray*}
Note that $\big\|A_{(0,0,\mu_h)}f-\frac{1}{2}f''\big\|_\infty=\sup_{x\in \R}\big|\frac{f(x+h)-f(x)-f'(x)h}{h^2}- \frac{1}{2}f''(x)\big|\to 0$
as $h\searrow 0$ for all $f\in \BUC^2(\R)$.
\end{example}

\begin{example}[Cauchy distributed jumps]\label{cex2}
 For $\gamma>0$ let $\delta_\gamma$ be given by
 \[
  \mu_\gamma\big((-\infty, b)\big):=\frac{\gamma}{\pi}\int_{-\infty}^b\frac{1}{y^2+\gamma^2}\, {\rm d}y=\frac{1}{2}+\arctan\bigg(\frac{b}{\gamma}\bigg)
 \]
 for $b\in \R$. Let $\Gamma\subset (0,\infty)$. Then, we have that
 \[
  \sup_{\gamma\in \Gamma}\int_{\R}1\wedge |y|^2\, {\rm d}\mu_\gamma (y)\leq \sup_{\gamma\in \Gamma}\mu_\gamma(\R)=1,
 \]
so that (A1) and (A2) are satisfied, but the first condition in \eqref{cond:NN} is violated. %in \cite{NutzNeuf}.
By Example \ref{compois} with $G=\R$ (using the notation from Example \ref{genlevy}), there exists a sublinear expectation space $(\Om,\FF,\EE)$ and an $\EE$-L\'evy process $(X_t)_{t\geq 0}$ such that for all $f\in \BUC(\R)$ the function
 \[
  u(t,x):=\EE(f(x+X_t)),\quad t\geq 0, x\in \R,
 \]
 is the unique classical solution of the fully nonlinear PIDE
 \begin{eqnarray*}
  u_t(t,x)&=&\sup_{\gamma\in \Gamma} \big(A_{0,0,\mu_\gamma} u(t)\big)(x), \quad (t,x)\in (0,\infty)\times \R,\\
  u(0,x)&=&f(x),\quad x\in \R
 \end{eqnarray*}
 with $u\in C^1([0,\infty);\BUC(\R))$.
\end{example}

\section{Proof of Theorem \ref{equivlevy}}\label{sec2}

\begin{proof}
Let $(\SS(t))_{t\geq 0}$ be a convex Markovian convolution semigroup which is continuous from above.
By Remark \ref{rem:kernel}, every $\SS(t)$ has a unique extension to a convex kernel
$\SS(t):C_b(G)\to C_b(G)$ which is continuous from above. Then the family
\[\EE_{s,t}(x,f):=(\SS(t-s)f)(x),\quad 0\le s < t\]
of convex kernels on $C_b(G)$ is continuous from above and satisfies the Chapman-Kolmogorov equations, i.e.
\[\EE_{s,t}(x,\EE_{t,u}(\cdot, f))=\EE_{s,u}(x,f)\] for all $0\le s < t < u$, $f\in C_b(G)$
and $x\in G$. Hence, it follows from \cite[Theorem 5.6]{DKN} that there exists a convex expectation space $(\Omega,\FF,\EE)$ and a family of random variables $X_t:\Omega\to G$, $t\ge 0$, such that \[\EE(f(X_0))=f(0)\] for all $f\in C_b(G)$ and
\[
\EE\big( f(X_{t_1},\dots, X_{t_n}, X_t)\big)=\EE \big(\EE_{s,t}\left(X_s,f(X_{t_1},\dots, X_{t_n},\, \cdot\, )\right)\big)
\]
for all $0\le s < t$,  $n\in\mathbb{N}$, $0\le t_1<\ldots < t_n\le s$, and $f\in C_b(G^{n+1})$.
Recall that $(x_{t_1},\dots x_{t_n},x_s)\mapsto\EE_{s,t}\left(x_s,f(x_{t_1},\dots, x_{t_n},\, \cdot\, )\right)$ is continuous, see e.g.~\cite[Proposition 5.5]{DKN} with $S=G$ and $T=G^n$. Next, we verify that $(X_t)_{t\ge 0}$ is an $\EE$-L\'evy process. For $f\in C_b(G)$ and $s,t\ge 0$ one has
\[
\EE\big(f(X_t)\big)=\EE\big(\EE_{0,t}(X_0,f)\big)=\EE\big((\SS(t) f)(X_0)\big)=(\SS(t)f)(0)
\]
and
\begin{align}
&\EE\big(f(X_{s+t}-X_s)\big)=\EE\big(\tilde f(X_s,X_{s+t})\big)=\EE\big(\EE_{s,s+t}(X_s,\tilde f(X_s,\, \cdot\, )\big)\nonumber\\
&=\EE\big((\SS(t)\tilde f(X_s,\, \cdot\, ))(X_s)\big)=\EE\big((\SS(t) f)(0)\big)=(\SS(t) f)(0),\label{prooflevy1}
\end{align}
where $\tilde f(x,y):=f(y-x)$ and $(\SS(t)\tilde f(x,\, \cdot\, ))(x)=(\SS(t)f_{-x})(x)=(\SS(t)f)(0)$ because $\SS(t)$ is a Markovian convolution. This shows that the random variables $X_{s+t}-X_s$ and $X_t$ have the same distribution under $\EE$. Moreover, for $s,t\ge 0$, $0\le t_1<\ldots<t_n\le s$ and $f\in C_b(G^{n+1})$, it follows by \eqref{prooflevy1} that
\[(\SS(t) f(x_{t_1},\ldots,x_{t_n},\,\cdot\, ))(0)=\EE\big(f(x_{t_1},\ldots,x_{t_n},X_{s+t}-X_s)\big)\]
for all $x_{t_1},\dots,x_{t_n}\in G$ and therefore,
\begin{align*}
&\EE\big(f(X_{t_1},\dots,X_{t_n},X_{s+t}-X_s)\big)=\EE\big(f_{-X_s}(X_{t_1},\dots,X_{t_n},X_{s+t})\big)\\
&=\EE\big(\EE_{s,s+t}\big(X_s,f_{-X_s}(X_{t_1},\dots,X_{t_n},\, \cdot\,)\big)\big)\\
&=\EE\big(\big(\SS(t) f_{-X_s}(X_{t_1},\dots,X_{t_n},\, \cdot\, )\big)(X_s)\big)\\
&=\EE\big(\big(\SS(t) f(X_{t_1},\dots,X_{t_n},\, \cdot\, )\big)(0)\big)\\
&=\EE\big(\EE\big(f(x_{t_1},\dots,x_{t_n},X_{s+t}-X_s)\big)|_{x_{t_1}=X_{t_1},\dots ,x_{t_n}=X_{t_n}}\big)
\end{align*}
which shows that $X_{s+t}-X_s$ is independent of $(X_{t_1},\dots,X_{t_n})$.
It remains to verify that $X_t\to X_0$ in distribution as $t\searrow 0$. To that end, fix $f\in C_b(G)$ and notice that there exists an increasing sequence $(f_n)$ in $\BUC(G)$ which converges pointwise to $f$. By Dini's theorem it follows that the continuous functions $t\mapsto\big(\SS(t)f_n\big)(0)$ converge uniformly on $[0,1]$ to the continuous function $t\mapsto\big(\SS(t)f\big)(0)$. Hence, for every $\varepsilon>0$ there exists some $n_0\in\N$ such that
\begin{align*}
  |\EE(f(X_t))-f(0)| &\leq \big|\big(\SS(t)f\big)(0)-\big(\SS(t)f_{n_0}\big)(0)\big|+\|\SS(t)f_{n_0}-f_{n_0}\|_\infty\\
  &\quad + |f_{n_0}(0)-f(0)|\\
  &\leq \varepsilon, \quad \text{as }t\searrow 0.
 \end{align*}
 Conversely, let $(X_t)_{t\geq 0}$ be an $\EE$-L\'evy process on a convex expectation space $(\Om,\FF,\EE)$.
Then, the family $(\SS(t))_{t\geq 0}$ defined by
\[
  \big(\SS(t)f\big)(x):=\EE(f(x+X_t))
 \]
 for $f\in \BUC(G)$ and $x\in G$ is a convex Markovian convolution semigroup. Indeed, for $f\in \BUC(G)$, $x\in G$ and $t\ge 0$, it holds that
\[\big(\SS(t)f_x\big)(0)=\EE(f(x+X_t))=\big(\SS(t)f\big)(x),\]
 so that $\SS(t)$ is a convex Markovian convolution which satisfies
\[\big(\SS(0)f\big)(x)=\EE(f(x+X_0))=\EE(f_x(X_0))=f_x(0)=f(x).\] Moreover, since $\EE(f(x+(X_{t+s}-X_s)))=\EE(f(x+X_t))=\big(\SS(t)f\big)(x)$
for all $s,t\geq 0$, $f\in \BUC(G)$, $x\in G$, and $X_{t+s}-X_s$ is independent of $X_s$ we obtain
 \begin{align*}
  \big(\SS(t+s)f\big)(0)=\EE(f(X_{t+s}))&=\EE(f(X_s+(X_{t+s}-X_s)))\\&=\EE\big(\big(\SS(t)f\big)(X_s)\big)=\big(\SS(s)\SS(t)f\big)(0).
 \end{align*}
Since $\SS(s)$ and $\SS(t)$ are Markovian convolutions we conclude the semigroup property $\SS(t+s)=\SS(s)\SS(t)$. It remains to show that $\lim_{t\searrow 0}\|\SS(t)f-f\|_\infty=0$ for all $f\in \BUC(G)$.
To do so, we first show that for every $c\geq 0$ and $\de>0$ we have that
 \begin{equation}\label{help1}
  \EE\big(c 1_{G\sm B(0,\de)}(X_t)\big)\to 0
 \end{equation}
 as $t\searrow 0$. Let $\varphi\colon G\to \R$ be defined by
 \[
  \varphi(y):=\frac{c}{\de} \big(d(y,0)\wedge \de\big)
 \]
 for $y\in G$. Then, $\varphi\in C_b(G)$ with $\varphi\geq 0$, $\varphi(0)=0$ and $\varphi(y)=c$ for all $y\in G\sm B(0,\de)$, so that
 \[
  0\leq \EE\big(c 1_{G\sm B(0,\de)}(X_t)\big)=\EE\big(\varphi(X_t)1_{G\sm B(0,\de)}(X_t)\big)\leq \EE(\varphi(X_t))\to 0
 \]
 as $t\searrow 0$. Now, fix $\ep >0$ and $f\in \BUC(G)$. Then there exists some $\de>0$ such that $|f(x+y)-f(x)|< \frac{\ep}{2}$
 for all $x,y\in G$ with $d(y,0)< \de$. For each $x\in G$ let $g_x\in \LL^\infty(G)$ be given by $y\mapsto 1_{B(0,\de)}(y)(f(x+y)-f(x))$. Then $\|g_x\|_\infty\leq \frac{\ep}{2}$ for all $x\in G$, and by \eqref{help1} there exists some $t_0>0$ such that
 \[
  \EE\big(2\|f\|_\infty 1_{G\sm B(0,\de)}(X_t)\big)\leq \frac{\ep}{2}
 \]
 for all $0<t<t_0$. Since $\EE\colon \LL^\infty(\Om,\FF)\to \R$ is $1$-Lipschitz continuous, we get
 \begin{align*}
  \big|\big(\SS(t)f\big)(x)-f(x)\big|&=\big|\EE\big(f(x+X_t)-f(x)\big)\big|\leq \EE\big(\left|f(x+X_t)-f(x)\right|\big)\\
  &\leq \EE\big(g_x(X_t)+2\|f\|_\infty 1_{G\sm B(0,\de)}(X_t)\big)\\
  &\leq \|g_x\|_\infty+\EE\big(2\|f\|_\infty  1_{G\sm B(0,\de)}(X_t)\big)\leq \ep
 \end{align*}
 for all $0<t<t_0$, which shows that $\|\SS(t)f-f\|_\infty\leq \ep$ for all $0<t<t_0$.

\end{proof}

\section{Randomizing linear semigroups and the proof of Theorem \ref{main}}\label{sec3}

Throughout this section all Markovian convolutions are defined on $\BUC(G)$. We assume that $\{A_\lambda: \lambda\in\Lambda\}$ is a given non-empty family of operators which satisfies the assumptions (A1) and (A2).
Recall that $A_\lambda$ generates a  Markovian convolution semigroup $(S_\lambda(t))_{t\ge 0}$ of linear operators, and the domain
 \[
 D=\bigg\{ f\in \bigcap_{\la\in \La}D(A_\la) : \big\{A_\la f\colon \la\in \La\big\}\;\text{is bounded and uniformly equicontinuous}\bigg\}
 \]
 is dense in $\BUC(G)$.

We consider finite partitions $P:= \{\pi\subset[0,\infty): 0\in\pi, \, \pi\text{ finite}\}$. For a partition $\pi=\{t_0,t_1,\dots,t_m\}\in P$ with $0=t_0< t_1< \ldots < t_m$ we set $|\pi|_\infty := \max_{j=1,\dots,m} (t_j-t_{j-1})$. The set of partitions with end-point $t$ is denoted by $P_t$, i.e.~$P_t := \{\pi \in P: \max \pi = t\}$.

For $f\in BUC(G)$ and $t\ge 0$, we  define
\[
 \big(J_t f\big)(x):=\sup_{\la\in \La} \big(S_\la(t)f\big)(x), \quad x\in G,
\]
and for a partition $\pi\in P$, we set
\[ J_\pi f := J_{t_1-t_0} \cdots J_{t_m-t_{m-1}} f,\]
where we assume that $\pi=\{t_0,t_1,\dots,t_m\}$ with $0=t_0< t_1< \ldots < t_m$, and $J_{\{0\}}f := f$.
Note that $J_t = J_{\{0,t\}}$ for $t>0$. Moreover, since $S_\la(t)$ is continuous from below for all $\la\in\Lambda$
and $t\ge 0$, it follows that $J_\pi$ is continuous from below for all $\pi\in P$.

Let $f\in D$. Then, by definition of $D$, the family $(A_\la f)_{\la\in \La}$ is bounded and, throughout the rest of this section, we denote
$$L_f := \sup_{\lambda\in\Lambda} \|A_\lambda f\|_{\infty} <\infty.$$

\begin{lemma}\label{lip12}
 \begin{enumerate}[a)]
 \item  $J_\pi$ is a sublinear Markovian convolution for all $\pi\in P$.
 \item Let $f\in D$. Then,
 \begin{align}
   \|J_{t_1} f - J_{t_2} f\|_\infty & \le L_f |t_1-t_2|, \quad t_1,t_2\ge 0,\label{rd01}\\
   \| J_\pi f - f\|_\infty & \le L_f \, t, \quad \pi\in P_t,\, t>0.\label{rd02}
 \end{align}
 \end{enumerate}
\end{lemma}

\begin{proof}
\begin{enumerate}[a)]
 \item Since $S_\lambda(t)$ is a linear Markovian convolution for all $\lambda\in\Lambda$, $J_t$ is a sublinear Markovian convolution for all $t\ge 0$. As this property is preserved under compositions, the same holds for $J_\pi$.

 \item Let $f\in D$. For $t_1,t_2\ge 0$, $x\in G$, and $\lambda_0\in\Lambda$ we have that
  \begin{align*}
  \big(S_{\la_0}(t_1)f\big)(x)-\big(J_{t_2}f\big)(x)&\leq \big(S_{\la_0}(t_1)f\big)(x)-\big(S_{\la_0}(t_2)f\big)(x)
\\&\leq \|S_{\la_0}(t_1)f-S_{\la_0}(t_2)f\|_\infty
  \\&\leq \sup_{\la\in \La}\|S_\la(t_1)f-S_\la(t_2)f\|_\infty
  \end{align*}
  and therefore, taking the supremum over $\la_0\in \La$,
  \[
   \big(J_{t_1}f\big)(x)-\big(J_{t_2}f\big)(x)\leq \sup_{\la\in \La}\|S_\la(t_1)f-S_\la(t_2)f\|_\infty.
  \]
By symmetry and taking the supremum over all $x\in G$, we thus get that
 \[
  \|J_{t_1}f-J_{t_2}f\|_\infty\leq \sup_{\la\in \La}\|S_\la(t_1)f-S_\la(t_2)f\|_\infty.
 \]
 Let $\la \in \La$ and w.l.o.g. let $t_1<t_2$. Then, as $f\in D(A_\lambda)$, it follows that (see, e.g.~\cite[Lemma~II.1.3]{MR1721989})
  \begin{align*}
  \|S_\la(t_1)f-S_\la(t_2)f\|_\infty&=\bigg\|\int_{t_1}^{t_2}S_\la(s)A_\la f\, {\rm d}s\bigg\|_\infty\leq \int_{t_1}^{t_2}\|S_\la(s)A_\la f\|_\infty\, {\rm d}s\\
  &\leq (t_2-t_1) \sup_{0\leq s\leq t}\|S_\la (s)A_\la f\|_\infty\leq (t_2-t_1)\|A_\la f\|_\infty\\
  &\leq L_f|t_1-t_2|
 \end{align*}
 for all $\la\in \La$. Here, we used the fact that $\|S_\la (s)g\|_\infty\leq \|g\|_\infty$ for all $\la\in \La$, $s\geq 0$ and $g\in \BUC(G)$. Taking the supremum over all $\la\in \La$, we obtain \eqref{rd01}.

 To show \eqref{rd02}, let $\pi=\{0, t_1,\dots, t_{m-1}, t\}\in P_t$ with $0< t_1<\ldots< t_{m-1}<t$ and $m\in \N$. Note that the case $\pi=\{0\}$ is trivial. For $m=1$, we have $\pi=\{0,t\}$, and by \eqref{rd01},
 \[ \|J_\pi f - f\|_\infty = \| J_t f - J_0 f\|_\infty \le L_f\, t.\]
 Now let $m\geq 2$, and set $\pi' := \{0,t_1,\dots,t_{m-1}\}$. By induction, we have $\|J_{\pi'} f -f \|_\infty\le L_f\, t_{m-1}$. As $J_{\pi'}$ is a sublinear Markovian convolution and therefore $1$-Lipschitz continuous, we obtain
 \begin{align*}
   \|J_\pi f -f\|_\infty & = \| J_{\pi'} J_{t-t_{m-1}} f - f\|_\infty \\
   & \le \| J_{\pi'} J_{t-t_{m-1}} f - J_{\pi'} f\|_\infty + \| J_{\pi'} f - f\|_\infty \\
   & \le  \|  J_{t-t_{m-1}} f -   f\|_\infty + \| J_{\pi'} f - f\|_\infty \\
   & \le L_f (t-t_{m-1}) + L_f\, t_{m-1} = L_f\, t.
 \end{align*}
  \end{enumerate}

 \vspace*{-1em}
\end{proof}

The following result shows that $J_\pi f$ depends continuously on the partition $\pi$.

\begin{lemma}\label{strongcontJ}
Let $m\in\N$, $\pi=\{0,t_1,\ldots,t_m\}\in P$ with $0<t_1<\ldots<t_m$, and for $n\in\N$ let $\pi_n =\{0,t_1^n,\dots, t_m^n\}\in P$ with $\lim_{n\to\infty} t_j^n = t_j$ for $j=1,\dots,m$. Then
\[
 \|J_{\pi}f-J_{\pi_n}f\|_\infty\to 0,\quad \mbox{as }n\to \infty,
\]
for all $f\in BUC(G)$.
\end{lemma}

\begin{proof}
Note that $0<t_1^n<\ldots<t_m^n$ for sufficiently large $n$. Fix $f\in \BUC(G)$.
We have to show the continuity of the map $(t_1,\dots,t_m)\mapsto J_{\{0,t_1,\dots,t_m\}}f$. By definition of $J_\pi$, it is  sufficient to show that the map
\[ [0,\infty)\to \BUC(G),\quad  t\mapsto J_tf\]
is continuous. Let $\ep>0$ and $t\ge 0$, and let $(t^n)$ be a sequence in
$[0,\infty)$ such that $t^n\to t$. By assumption (A2) there exists $f_0\in D$ with $\|f-f_0\|_\infty\leq \frac{\ep}{3}$. Since $J_{s}=J_{0,s}$ is a sublinear Markovian convolution by Lemma 5.1 a), it is $1$-Lipschitz and it follows that
\[
 \|J_s f -J_s f_0\|_\infty\leq \|f-f_0\|_\infty<\frac{\ep}{3}
\]
for all $s\geq 0$. Hence, it follows from  Lemma~\ref{lip12} that
  \begin{align*}
   \|J_t f -J_{t^n}f\|_\infty&\leq \|J_t f -J_t f_0\|_\infty+ \| J_t f_0 - J_{t^n} f_0\|_\infty + \|J_{t^n} f -J_{t^n}f_0\|_\infty
   \\&\leq \tfrac 23 \ep +L_{f_0} |t-t^n|<\ep
  \end{align*}
 for sufficiently large $n\in\N$.
\end{proof}

The above results allow to consider the limit of $J_\pi f$ when the mesh size of the partition tends to zero. To that end, we first note that for $s,t\geq 0$, $f\in \BUC(G)$ and $x\in G$, it holds
\begin{align*}
 \big(J_{t+s}f\big)(x)&=\sup_{\la\in \La} \big(S_\la(t+s)f\big)(x)=\sup_{\la\in \La} \big(S_\la(t)S_\la(s)f\big)(x)\\
 &\leq \sup_{\la\in \La} \big(S_\la(t)J_{s}f\big)(x)=\big(J_{t}J_{s}f\big)(x).
\end{align*}
From this, we obtain for $\pi_1,\pi_2\in P_t$ with $\pi_1\subset \pi_2$  the pointwise inequality
\begin{equation}\label{monotone}
 J_{\pi_1}f\leq J_{\pi_2}f.
\end{equation}
In particular, for $\pi_1,\pi_2\in P_t$ we have that
\begin{equation}\label{0}
 \big(J_{\pi_1}f\big)\vee \big(J_{\pi_2}f\big)\leq J_\pi f,
\end{equation}
where $\pi:=\pi_1\cup \pi_2\in P_t$. Therefore, we define
\begin{equation}\label{def:S}
 \big(\SS(t)f\big)(x):=\sup_{\pi\in P_t} \big(J_\pi f\big)(x)
\end{equation}
for all $t\geq 0$, $x\in G$ and $f\in \BUC(G)$. Note that $\SS(0)f=f$ for all $f\in \BUC(G)$.

\begin{lemma}\label{lip1}
$\SS(t)$ is a sublinear Markovian convolution for all $t\ge 0$, which satisfies
\[\|\SS(t)f-f\|_\infty\leq L_f\, t,\quad  f\in D.\]
Moreover,
\[    \lim_{t\searrow 0}\|\SS(t)f-f\|_\infty = 0\]
for all $f\in BUC(G)$.
\end{lemma}

\begin{proof}
As $J_\pi$ is a sublinear Markovian convolution for all $\pi\in P_t$, the same holds for $\SS(t)$. For $f\in D$, $x\in G$ and  $\varepsilon>0$ there exists $\pi_0\in P_t$ such that
\[\big(\SS(t)f\big)(x)-f(x)\le \big(J_{\pi_0} f\big)(x)-f(x)+\varepsilon\le\sup_{\pi\in P_t}\|J_\pi f-f\|_\infty+\varepsilon\]
and
\[f(x)-\big(\SS(t)f\big)(x)\le f(x)-\big(J_{\pi_0} f\big)(x)\le\sup_{\pi\in P_t}\|J_\pi f-f\|_\infty. \]
By Lemma \ref{lip12} it follows that
\[
  \|\SS(t)f-f\|_\infty\leq \sup_{\pi\in P_t}\|J_\pi f-f\|_\infty \leq L_f\, t.
 \]
From this, we obtain that $\lim_{t\searrow 0}\|\SS(t)f-f\|_\infty= 0$ for $f\in \BUC(G)$
with the same density argument as in the proof of Lemma~\ref{strongcontJ}.
\end{proof}

\begin{lemma}\label{monlimit}
 For $t\ge0$, let $(\pi_n)_{n\in\N}$ be a sequence in $P_t$ such that  $\pi_n\subset \pi_{n+1}$ for all $n\in \N$, and $\lim_{n\to\infty} |\pi_n|_\infty= 0$. Then
 \[
  J_{\pi_n}f\nearrow \SS(t)f, \quad \mbox{as }n\to \infty,
 \]
 for all $f\in \BUC(G)$.
\end{lemma}

\begin{proof}
Fix  $f\in \BUC(G)$. For $t=0$, the statement is trivial. For   $t>0$ and $x\in G$ we define
  \[
  \big(J_\infty f\big)(x):=\sup_{n\in \N} \big(J_{\pi_n}f\big)(x)
 \]
Then, $J_\infty$ is a sublinear Markovian convolution. Since $\pi_n\subset \pi_{n+1}$, it follows from \eqref{monotone}
that
 \[
  J_{\pi_n}f\nearrow J_\infty f, \quad \mbox{as }n\to \infty.
 \]
By definition of $\SS(t)$, it clearly holds $J_\infty f \leq \SS(t)f$.
As for the other inequality, let $\pi=\{t_0,t_1,\ldots, t_m\}\in P_t$ with $m\in \N$ and $0=t_0<t_1<\ldots <t_m=t$. Since $|\pi_n|_\infty\searrow 0$ as $n\to \infty$, we may w.l.o.g. assume that $\#\pi_n\geq m+1$ for all $n\in \N$. Moreover, we can choose  $0=t_0^n<t_1^n<\ldots <t_m^n=t$ with $\pi_n':=\{t_0^n,t_1^n,\ldots, t_m^n\}\subset \pi_n$ and $\lim_{n\to\infty} t_i^n = t_i$ for all $i=1,\dots,m-1$. Then, by Lemma \ref{strongcontJ}, we have that
 \[
  \|J_\pi f-J_{\pi_n'}f\|_\infty\to 0, \quad \mbox{as }n\to \infty.
 \]
 Since
 \[
  J_\infty f\geq J_{\pi_n}f\geq J_{\pi_n'}f\geq J_\pi f-\|J_\pi f-J_{\pi_n'}f\|_\infty
 \]
 we obtain that $J_\infty f\geq J_\pi f$.
 Taking the supremum over all $\pi\in P_t$, we thus get that $J_\infty f=\SS(t)f$.
 %Now, let $f\in C_b(G)$ with $f\geq 0$. Then, by the assumptions on $G$, there exists a sequence $(\varphi_n)_{n\in \N}\subset C_0(G)$ with $\varphi_n\nearrow f$ as $n\to %\infty$. As $J_\infty$ and $\SS(t)$ are both continuous from below, we get that
 %\[
 % J_\infty f=\lim_{n\to \infty} J_\infty \varphi_n=\lim_{n\to \infty} \SS(t)\varphi_n=\SS(t)f.
 %\]
 %Now, let $f\in C_b(G)$. Then, as $J_\infty$ and $\SS(t)$ are both sublinear and Markovian, we get that
 %\[
 % J_\infty f=J_\infty\big(f+\|f\|_\infty\big)-\|f\|_\infty=\SS(t)\big(f+\|f\|_\infty\big)-\|f\|_\infty=\SS(t)f
 %\]
 %and therefore, $J_{\pi_n}f\nearrow \SS(t)f$ as $n\to \infty$.
\end{proof}

\begin{corollary}\label{seq}
 Let $t\ge 0$. Then, there exists a sequence $(\pi_n)$ in $P_t$ such that
 \[
  J_{\pi_n}f\nearrow \SS(t)f,\quad\mbox{as } n\to\infty.
 \]
Moreover,
\begin{equation}\label{rd03}
  \SS(t)f=\sup_{n\in \N} \big(J_{\frac{t}{n}}\big)^n f=\sup_{n\in\N}\big(J_{2^{-n}t}\big)^{2^n}f=\lim_{n\to \infty}\big(J_{2^{-n}t}\big)^{2^n}f,
 \end{equation}
for all $f\in \BUC(G)$,  where the supremum is understood pointwise.
\end{corollary}

\begin{proof}
Choose $\pi_n :=\big\{\frac{kt}{2^n}\, :\, k\in \{0,\ldots ,2^n\}\big\}$ in Lemma \ref{monlimit} to obtain the first statement. In particular,
\begin{equation}\label{revised1}
  \SS(t)f=\SS(t)f=\sup_{n\in\N} J_{\pi_n} f=\sup_{n\in\N}\big(J_{2^{-n}t}\big)^{2^n}f=\lim_{n\to \infty}\big(J_{2^{-n}t}\big)^{2^n}f.
\end{equation}
For $\tilde \pi_n:=\big\{\frac{kt}{n}\, :\, k\in \{0,\ldots, n\}\big\}$ it holds
$\tilde\pi_{2^n}= \pi_{n}$ for all $n\in\N$. Therefore, by \eqref{revised1}, it follows that
\[
 \SS(t)f=\sup_{n\in\N} J_{\pi_n} f=\sup_{n\in\N} J_{\tilde\pi_{2^n}} f\leq \sup_{n\in\N}J_{\tilde\pi_{n}}f\leq \SS(t)f.
\]
Hence,
\[
 \SS(t)f=\sup_{n\in\N}J_{\tilde\pi_{n}}f=\sup_{n\in \N} \big(J_{\frac{t}{n}}\big)^n f,
\]
which yields the first equality in \eqref{rd03} and therefore the assertion.
\end{proof}

\begin{proposition}[Dynamic programming principle]\label{dpp}
$(\SS(t))_{t\geq 0}$ is a Markovian convolution semigroup of sublinear operators. In particular, for every
$s,t\geq 0$ one has
\begin{equation}\label{rd04}
 \SS(s+t)f=\SS(s)\SS(t)f
\end{equation}
for every $f\in BUC(G)$.
\end{proposition}

\begin{proof}
Let $\pi_0\in P_{s+t}$ and $\pi:=\pi_0\cup \{s\}$. Then, $\pi\in P_{s+t}$ with $\pi_0\subset \pi$, and by \eqref{monotone} we get that
 \[
  J_{\pi_0} f\leq J_{\pi}f.
 \]

We have already shown all properties of a Markovian convolution semigroup except the semigroup property  \eqref{rd04}, i.e.~the dynamic programming principle.

 If $s=0$ or $t=0$ the statement is trivial. Therefore, let $s,t>0$.
 Let $m\in \N$ and $0=t_0<t_1<\ldots <t_m=s+t$ with $t_i=s$ for some $i\in \{1,\ldots, m\}$, and define
 $\pi:=\{t_0,\ldots, t_m\}\in P_{s+t}$. Then, for $\pi_1:=\{t_0,\ldots, t_i\}\in P_s$ and $\pi_2:=\{t_i-s,\ldots, t_m-s\}\in P_t$ we have
 \[
  J_{\pi_1}=J_{t_1-t_0}\cdots J_{t_i-t_{i-1}},\quad\mbox{and}\quad J_{\pi_2}=J_{t_{i+1}-t_i}\cdots J_{t_m-t_{m-1}},
 \]
 and therefore
 \begin{align*}
  J_\pi f&=J_{t_1-t_0}\cdots J_{t_m-t_{m-1}}f=\big(J_{t_1-t_0}\cdots J_{t_i-t_{i-1}}\big)\big(J_{t_{i+1}-t_i}\cdots J_{t_m-t_{m-1}}f\big)\\
  &=J_{\pi_1}J_{\pi_2}f\leq J_{\pi_1}\SS(t)f\leq \SS(s)\SS(t)f.
 \end{align*}
 Taking the supremum over all $\pi\in P_{s+t}$, we get $\SS(s+t)f\leq \SS(s)\SS(t)f$.
 Conversely, by Corollary \ref{seq} there exists  a sequence $(\pi_n)$ in $P_t$ such that $J_{\pi_n}f\nearrow \SS(t)f$ as $n\to \infty$. For $\pi_0\in P_s$ and $\pi_n':=\pi_0\cup \{r+s : r \in \pi_n\}\in P_{s+t}$ it holds
 $J_{\pi_n'}=J_{\pi_0}J_{\pi_n}$. Since $J_{\pi_0}$ is continuous from below we have
 \[
  J_{\pi_0}\SS(t)f=\lim_{n\to \infty}J_{\pi_0}J_{\pi_n}f=\lim_{n\to \infty}J_{\pi_n'}f\leq \SS(s+t)f.
 \]
 Taking the supremum over all $\pi_0\in P_s$, we get that $\SS(s)\SS(t)f\leq \SS(s+t)f$.
\end{proof}

% Anpassen
The following lemma is a special case of Jensen's inequality for vector valued functions. For the reader's convenience we provide a short proof and refer to \cite[Section 1.2.]{hytonen16} for an introduction to Bochner integration.
\begin{lemma}\label{intsemi}
 Let $\SS\colon \BUC(G)\to \BUC(G)$ be convex and continuous. Let $(\Om,\FF,\nu)$ be a finite measure space with $\nu\neq 0$. Further, let $g\colon \Om\to \BUC(G)$ be bounded and $\FF$-$\BB(\BUC(G))$-measurable with separable range $g(\Om)$, i.e.~$g$ is Bochner integrable, such that $\SS g\colon \Om\to \BUC(G), \; \om\mapsto \SS \big(g(\om)\big)$ is again bounded. Then, $\SS g$ is Bochner integrable and we have that
 \[
  \SS\left(\frac{1}{\nu(\Om)}\int_\Om g\, {\rm d}\nu\right)\leq \frac{1}{\nu(\Om)}\int_\Om \SS g \, {\rm d}\nu.
 \]
\end{lemma}

\begin{proof}
 Since $\SS$ is continuous, we obtain that $\SS g:\Omega\to\BUC(G)$ is $\FF$-$\BB(\BUC(G))$-measurable with separable range $(\SS g)(\Om)$ and thus Bochner integrable. If $g$ is a simple function, $\SS g$ is a simple function and the assertion follows by convexity of $\SS$. Since $g$ is $\FF$-$\BB(\BUC(G))$-measurable with separable range $g(\Om)$, there exists a sequence of simple functions $(g_n)_{n\in \N}$ with  $\lim_{n\to\infty}\|g_n(\om)-g(\om)\|_\infty= 0$ for all $\om\in \Om$. By continuity of $\SS$, we obtain $\lim_{n\to\infty}\|\SS g_n(\om)-\SS g(\om)\|_\infty=0$ for all $\om\in \Om$. Hence, by definition of Bochner's integral it follows that
 \begin{align*}
 \SS\left(\frac{1}{\nu(\Om)}\int_\Om g\, {\rm d}\nu\right)&=\lim_{n\to \infty} \SS\left(\frac{1}{\nu(\Om)}\int_\Om g_n\, {\rm d}\nu\right)\\&\leq \lim_{n\to \infty} \frac{1}{\nu(\Om)}\int_\Om \SS g_n\, {\rm d}\nu=\frac{1}{\nu(\Om)}\int_\Om \SS g\, {\rm d}\nu.
 \end{align*}
\end{proof}

Fix $f\in D$. Since $\{A_\la f\colon \la\in \La\}\subset \BUC(G)$ is bounded and uniformly equicontinuous, it follows that
\[
 \AA f:=\sup_{\la \in \La} A_\la f\in \BUC(G),
\]
where the supremum is understood pointwise.

\begin{lemma}\label{intsemi1}
Fix $f\in D$. Then, for $\pi\in P$ and $t>0$ it holds
 \[
  J_\pi f-f\leq \int_0^{\max \pi} \SS(s)\AA f \, {\rm d}s \quad\mbox{and}\quad \SS(t)f-f\leq \int_0^t \SS(s)\AA f \, {\rm d}s .
 \]
 \end{lemma}

\begin{proof}
 Since $\AA f\in \BUC(G)$, the mapping $[0,\infty)\to \BUC(G)$, $s\mapsto \SS(s)\AA f$ is continuous. Therefore, the Bochner integrals are well-defined. Then, for all $t,h>0$ we have
 \begin{align}\label{intsemieq}
  J_h f-f&=\sup_{\la \in \La} S_\la(h)f-f=\sup_{\la \in \La} \int_0^h S_\la(s)A_\la f\, {\rm d} s \nonumber \\& \leq \int_0^h\SS(s)\AA f\,{\rm d} s=\int_t^{t+h}\SS(s-t)\AA f\, {\rm d}s.
 \end{align}

We prove the first inequality by induction on $m=\# \pi$. If $m=1$, i.e.~if $\pi=\{0\}$, the statement is trivial. Hence, assume that
  \[
  J_{\pi'} f-f\leq \int_0^{\max \pi'} \SS(s)\AA f \, {\rm d}s
 \]
 for all $\pi'\in P$ with $\#\pi'=m$ for some $m\in \N$. Let $\pi=\{t_0,t_1,\ldots, t_m\}$ with $0=t_0< t_1< \ldots < t_m$ and $\pi':=\pi\sm \{t_m\}$. Then, it follows from \eqref{intsemieq} and Lemma \ref{intsemi} that
 \begin{align*}
  J_\pi f -J_{\pi'}f&\leq J_{\pi'}\big(J_{t_m-t_{m-1}}f-f\big)\leq J_{\pi'}\Big(\int_{t_{m-1}}^{t_m} \SS(s-t_{m-1})\AA f\, {\rm d}s\Big)\\
  &\leq \SS(t_{m-1})\Big(\int_{t_{m-1}}^{t_m} \SS(s-t_{m-1})\AA f\, {\rm d}s\Big)\leq\int_{t_{m-1}}^{t_m}\SS(s)\AA f\, {\rm d}s,
 \end{align*}
 where the first inequality follows from the sublinearity of $J_{\pi'}$.
 By induction, we thus get
 \begin{align*}
 J_\pi f -f&= \big(J_\pi f -J_{\pi'}f\big)+\big(J_{\pi'} f-f\big)\leq \int_{t_{m-1}}^{t_m}\SS(s)\AA f\, {\rm d}s+\int_{0}^{t_{m-1}}\SS(s)\AA f\, {\rm d}s\\
 &=\int_0^{\max \pi}\SS(s)\AA f\, {\rm d}s.
 \end{align*}

In particular, for every $\pi\in P_t$ it holds
 \[
  J_\pi f-f\leq \int_0^t\SS(s)\AA f\, {\rm d}s.
 \]
Taking the supremum over all $\pi\in P_t$ yields the second assertion.
\end{proof}

\begin{lemma}\label{appen3}
 Let $M\subset \BUC(G)$ be bounded and uniformly equicontinuous. Then,
 \[
  \sup_{\la\in \La}\sup_{g\in M}\|S_\la(t)g-g\|_\infty\to 0, \quad \mbox{as }t\searrow 0.
 \]
\end{lemma}

\begin{proof}
 Let $\ep>0$ and $C:=\sup_{g\in M}\|g\|_\infty$. Then, there exists some $\de>0$ such that
 \[
  \sup_{g\in M}|g(x)-g(y)|\leq \ep
 \]
 for all $x,y\in G$ with $d(x,y)\leq \de$. Let $\varphi(y):=\frac{1}{\de}\big(d(y,0)\wedge \de\big)$ for all $y\in G$. Then, $\varphi\in \BUC(G)$, $0\leq \varphi\leq 1$, $\varphi(0)=0$ and $\varphi(y)=1$ for all $y\in G\sm B(0,\de)$.
Fix $\la\in \La$, $g\in M$ and $x\in G$. Since $g(x+\cdot )-g(x)\le\varepsilon+2C\varphi$, we get
 \[
  \big|\big( S_\la(t)g\big)(x)-g(x)\big|=\big|\big[S_\la(t)\big(g(x+\cdot )-g(x)\big)\big](0)\big|\leq \ep +2C\big(\SS(t)\varphi\big)(0),
 \]
 so that
 \[
  \sup_{\la\in \La}\sup_{g\in M}\|S_\la(t)g-g\|_\infty\leq \ep +2C\big(\SS(t)\varphi\big)(0).
 \]
 Since $\big(\SS(t)\varphi\big)(0)\to 0$ as $t\searrow 0$ by Lemma \ref{lip1}, we obtain the assertion.
\end{proof}

\begin{lemma}\label{generator1}
 For every $f\in D$ one has
 \[
  \lim_{h\searrow 0}\bigg\|\frac{\SS(h)f-f}{h}- \AA f\bigg\|_\infty=0.
 \]
\end{lemma}

\begin{proof}
 Let $\ep>0$. Then, by Lemma \ref{appen3} and Lemma \ref{lip1} there exists $h_0>0$ such that
 \[
  S_\la(h) A_\la f-A_\la f\geq -\ep\quad\mbox{for all } \la \in \La
 \]
 and
 \[
  \SS(h)\AA f-\AA f\leq\ep
 \]
 for all $0<h\leq h_0$. Then, for every $0<h\leq h_0$ and $\lambda\in\Lambda$ we get
 \[
  \SS(h)f-f \geq S_\la(h)f-f =\int_0^h S_\la(s) A_\la f\, {\rm d}s \geq \big(A_\la f-\ep \big)h
 \]
 so that
 \begin{equation}\label{gen1}
  \frac{\SS(h)f-f}{h}\geq \AA f-\ep.
 \end{equation}
 On the other hand, it follows from Lemma \ref{intsemi1} that for $0<h\le h_0$
 \[
  \SS(h)f-f-h\AA f\leq \int_0^h \SS(s)\AA f\, {\rm d}s-h\AA f=\int_0^h \SS(s)\AA f-\AA f\, {\rm d}s\leq h\ep,
 \]
 which yields
 \[
  \frac{\SS(h)f-f}{h}-\AA f\leq  \ep.
 \]
 Together with \eqref{gen1}, we obtain
 \[
  \bigg\|\frac{\SS(h)f-f}{h}-\AA f\bigg\|_\infty\leq \ep
 \]
 for all $0<h\leq h_0$.
\end{proof}

% Allgemein mit HG und Generator
\begin{proposition}\label{viscosity}
 For $f\in \BUC(G)$ the function \[u(t,x)=\big(u(t)\big)(x)=\big(\SS(t)f\big)(x),\quad t\ge 0, \,x\in G,\]
 is a $D$-viscosity solution of the fully nonlinear PDE
\begin{eqnarray*}
  u_t(t,x)&=&\sup_{\la\in \La} \big(A_\la u(t)\big)(x), \quad (t,x)\in (0,\infty)\times G,\\
  u(0,x)&=&f(x),\quad x\in G.
 \end{eqnarray*}
\end{proposition}

\begin{proof}
Fix $t>0$ and $x\in G$. We first show that $u$ is a $D$-viscosity subsolution. Let $\psi\colon (0,\infty)\to \BUC(G)$ differentiable with $\big(\psi(t)\big)(x)=\big(u(t)\big)(x)$, $\psi(s)\geq u(s)$ and $\psi(s)\in D$ for all $s>0$. Then, for every $h\in (0,t)$, it follows from
Proposition \ref{dpp} that
\begin{align*}
 0&=\frac{\SS(h)\SS(t-h)f-\SS(t)f}{h}=\frac{\SS(h)u(t-h)-u(t)}{h}\\
 &\leq \frac{\SS(h)\psi(t-h)-u(t)}{h} \leq \frac{\SS(h)\big(\psi(t-h)-\psi(t)\big)+\SS(h)\psi(t)-u(t)}{h}\\
 &= \SS(h)\bigg(\frac{\psi(t-h)-\psi(t)}{h}\bigg)+\frac{\SS(h)\psi(t)-\psi(t)}{h}+\frac{\psi(t)-u(t)}{h}.
\end{align*}
Let $\ep>0$. Then, by Lemma \ref{generator1} and Lemma \ref{lip1}, there exists $0<h_0<t$ such that
for all $0<h<h_0$ one has
\[
 \frac{\SS(h)\psi(t)-\psi(t)}{h}\leq \AA\psi(t)+\frac{\ep}{3},\quad
 \frac{\psi(t-h)-\psi(t)}{h}\leq -\psi_t(t)+\frac{\ep}{3},
\]
and
\[
 \SS(h)\big(-\psi_t(t)\big)\leq -\psi_t(t)+\frac{\ep}{3}.
\]
Hence, we get
\begin{align*}
 0&\leq \SS(h)\big(-\psi_t(t)\big)+\AA\psi(t)+\frac{2\ep}{3}+\frac{\psi(t)-u(t)}{h}\\&\leq -\psi_t(t)+\AA\psi(t)+\ep+\frac{\psi(t)-u(t)}{h}
\end{align*}
for all $0<h<h_0$. Since $\big(\psi(t)\big)(x)=\big(u(t)\big)(x)$ we obtain that
\[
 0\leq -\big(\psi_t(t)\big)(x)+\big(\AA\psi(t)\big)(x)+\ep.
\]
Letting $\ep\searrow 0$ yields $\big(\psi_t(t)\big)(x) \leq \big(\AA\psi(t)\big)(x)$.

To show that $u$ is a $D$-viscosity supersolution, let $\psi\colon (0,\infty)\to \BUC(G)$ differentiable with $\big(\psi(t)\big)(x)=\big(u(t)\big)(x)$, $\psi(s)\in D$ and $\psi(s)\leq u(s)$ for all $s>0$. By Proposition \ref{dpp}, for all $h>0$ with
$0<h<t$ we get
\begin{align*}
 0&=\frac{\SS(t)f-\SS(h)\SS(t-h)f}{h}=\frac{u(t)-\SS(h)u(t-h)}{h}\leq \frac{u(t)-\SS(h)\psi(t-h)}{h}\\
 &= \frac{u(t)-\psi(t)}{h}+\frac{\psi(t)-\SS(h)\psi(t)}{h}+\frac{\SS(h)\psi(t)-\SS(h)\psi(t-h)}{h}\\
 &\leq \frac{u(t)-\psi(t)}{h}+\frac{\psi(t)-\SS(h)\psi(t)}{h}+\SS(h)\bigg(\frac{\psi(t)-\psi(t-h)}{h}\bigg)
\end{align*}
Let $\ep>0$. Then, by Lemma \ref{generator1} and Lemma \ref{lip1} there exists $0<h_0<t$ such that
\[
 \frac{\psi(t)-\SS(h)\psi(t)}{h}\leq -\AA\psi(t)+\frac{\ep}{3},
\]
\[
 \frac{\psi(t)-\psi(t-h)}{h}\leq \psi_t(t)+\frac{\ep}{3},
\]
and
\[
 \SS(h)\big(\psi_t(t)\big)\leq \psi_t(t)+\frac{\ep}{3}
\]
for all $0<h<h_0$. We thus get
\[
 0\leq \frac{u(t)-\psi(t)}{h}-\AA\psi(t)+\SS(h)\psi_t(t)+\frac{2\ep}{3}\leq \frac{u(t)-\psi(t)}{h}-\AA\psi(t)+\psi_t(t)+\ep
\]
for all $0<h<h_0$. Since $\big(\psi(t)\big)(x)=\big(u(t)\big)(x)$ we obtain that
\[
 0\leq -\big(\AA\psi(t)\big)(x)+\big(\psi_t(t)\big)(x)+\ep.
\]
Letting $\ep\searrow 0$ yields $\big(\psi_t(t)\big)(x) \geq \big(\AA\psi(t)\big)(x)$.
\end{proof}

In order to complete the proof of Theorem \ref{main}, it remains to show that there exists a set $\PP$ of probability measures
and a stochastic process $(X_t)_{t\geq 0}$ on a measurable space $(\Om,\FF)$  such that the viscosity solution in Proposition \ref{viscosity} is of the form
\[
 u(t,x)=\sup_{\P\in \PP} \E_\P(f(x+X_t))
\]
for $t\geq 0$ and $x\in G$. This is shown in the following proposition for which the
assumption (A2) is not needed.

\begin{proposition}\label{dualrep}
 There exists a set $\PP$ of probability measures on a measurable space $(\Om,\FF)$ and an $\EE$-L\'evy process $(X_t)_{t\geq 0}$ such that
 \[
  \big(\SS(t)f\big)(x)=\EE(f(x+X_t))=\sup_{\P\in \PP} \E_\P(f(x+X_t))
 \]
 for all $f\in \BUC(G)$, $t\ge 0$ and $x\in G$, where $\EE(Y)=\sup_{\P\in \PP}\E_\P(Y)$, $Y\in \LL^\infty(\Om,\FF)$.
\end{proposition}

\begin{proof}
 Let $(\Om_0,\FF_0,\Q)$ be a probability space such that there exists an independent family of L\'evy processes $(X^\la)_{\la\in \La}$, where $X^\la$ is a L\'evy process with generator $A_\la$, i.e.
 \[
  \E_\Q(f(x+X_t^\la))=\big(S_\la(t)f\big)(x)
 \]
 for all $x\in G$, $t\geq 0$ and $f\in \BUC(G)$. We call $\phi=(\phi_t)_{t\geq 0}$ $\La$-simple if there exist $k\in \N$ and $0=t_0<t_1<\ldots< t_k$ such that
 \begin{equation}\label{simpleproc}
  \phi_t=\sum_{j=0}^{k-1}\phi^{(j)}1_{(t_j,t_{j+1}]}(t)+\phi^{(k)}1_{(t_k,\infty)}(t),
 \end{equation}
 where for $j\in \{0,\ldots, k\}$ the mapping
 \[
  \phi^{(j)}\colon G\to\La
 \]
 is measurable with finite range. For such a $\phi$ we define $X^\phi_{t_0}=X^\phi_0:=0$ and inductively, for $j\in \{0,\ldots, k-1\}$, we then define $\la_j:=\phi^{(j)}(X^\phi_{t_j})$ and
\[
 X^\phi_t:=X^\phi_{t_j}+X^{\la_j}_t-X^{\la_j}_{t_j}
\]
for all $t_j<t\leq t_{j+1}$. Finally, we define $\la_k:=\phi^{(k)}(X^\phi_{t_k})$ and
\[
 X^\phi_t:= X^\phi_{t_k}+X^{\la_k}_t-X^{\la_k}_{t_k}
\]
for $t>t_k$. Then, $X^\phi=(X^\phi_t)_{t\geq 0}$ can be interpreted as a stochastic integral w.r.t. the $\La$-simple process $\phi$. Let $h_1,h_2>0$. Then,
 \[
  \big(J_{h_1}J_{h_2}f\big)(x)=\sup_{\eta\in \QQ}\; \int_G f(x+\, \cdot\, )\, {\rm d}\eta
 \]
 for $f\in \BUC(G)$, where $\QQ$ is the set of probability measures of the form
 \[
  \eta \colon \BB \to [0,1],\quad B\mapsto \bigg[S_{\la_0}(h_1)\bigg(\sum_{k=1}^n1_{B_k}S_{\la_k}(h_2)1_B\bigg)\bigg](0)
 \]
with $k\in \N$, $\la_0,\la_1,\ldots,\la_n\in \La$ and $B_1\ldots, B_n\in \BB$ is a measurable partition of $G$. Here, we identify $S_\la(h)$ for $\la\in \La$ and $h>0$ with the unique translation invariant kernel associated to it. Inductively, we therefore obtain that $J_\pi$ with $\pi=\{t_0,\ldots, t_k\}$ admits a dual representation in terms of distributions of stochastic integrals w.r.t. $\La$-simple processes of the form \eqref{simpleproc}. For $t\geq 0$ and $f\in \BUC(G)$, we thus obtain that
\[
 \big(\SS(t)f\big)(0)=\sup_{\phi\; \La\text{-simple}} \E_\Q(f(X_t^\phi)).
\]
Now, consider $\Om:=G^{[0,\infty)}$, the product $\si$-algebra $\FF$, the canonical process $(X_t)_{t\geq 0}$ and
\begin{equation}\label{rd05}
 \PP:=\big\{\Q\circ \big(X^{\phi}\big)^{-1}\, \big|\, \phi\;\text{is $\La$-simple}\big\}.
\end{equation}
Then, for $\EE:=\sup_{\P\in \PP}\E_\P(\, \cdot \, )$ we have that
\[
 \big(\SS(t)f\big)(x)=\EE(f(x+X_t))=\sup_{\P\in \PP} \E_\P(f(x+X_t))
\]
for all $t\geq 0$ and $f\in \BUC(G)$. Then, $(X_t)_{t\geq 0}$ satisfies all properties in the definition of an $\EE$-L\'evy process. Note that (iii) and (iv) follow immediately from the dual representation of the constructed expectation $\EE(\cdot)=\sup_{\P\in \PP}\E_\P(\cdot)$ and (v) follows as in the proof of Theorem \ref{equivlevy}.
\end{proof}

The construction of the last proposition allows us to finish the proof of Theorem~\ref{main}.

\begin{proof}[Proof of Theorem \ref{main}]
 By Proposition \ref{dualrep}, there exists a sublinear expectation space $(\Om,\FF,\EE)$ and an $\EE$-L\'evy process $(X_t)_{t\geq 0}$ such that
  \[
  \big(\SS(t)f\big)(x)=\EE(f(x+X_t))
 \]
 for all $t\geq 0$, $x\in G$ and $f\in \BUC(G)$. By Proposition \ref{viscosity},
 $$u(t,x):=\big(\SS(t)f\big)(x)=\EE(f(x+X_t))$$
 is a $D$-viscosity solution to \eqref{PDE1} - \eqref{PDE2}. Moreover, by Proposition \ref{dualrep} there exists a set of probability measures $\PP$ on $(\Om,\FF)$ such that $\EE(Y)=\sup_{\P\in \PP}\E_\P(Y)$ for all $Y\in \LL^\infty(\Om,\FF)$.
\end{proof}

\begin{remark}
 a)  We note that the set $\PP$ is constructed in \eqref{rd05}, based on $\Lambda$-simple functions. This description is in some sense different from the approach by Neufeld and Nutz \cite{NutzNeuf} and K\"uhn \cite{k18}, where the set of measures is a priori given in relation to the family of L\'evy triplets. The question of characterizing the set $\PP$ in the general situation considered in Theorem~\ref{main} seems to be hard and is outside the scope of the present paper.

 b) The construction of the nonlinear semigroup is similar to Nisio's approach \cite{Nisio}, where only equidistant partitions of the time interval are used. As we have shown in Lemma~\ref{monlimit}, $\SS(t)$ is the limit of $J_{\pi_n}$ for any increasing sequence $(\pi_n)_{n\in\N}$ with $|\pi_n|_\infty\to 0$. Taking equidistant partitions, we obtain the analogue of Nisio's semigroup. However, in Nisio \cite{Nisio} strongly continuous semigroups on $L^\infty(G)$ were considered, and it is now well known that such semigroups have a bounded generator (\cite[Corollary 4.3.19]{MR2798103}), which does not cover the interesting cases. Therefore, we work with $\BUC (G)$ as the basic space. We plan to consider other relevant spaces like $L^p$ or the space of bounded continuous functions in the future.
\end{remark}

We finish with the proof of Proposition \ref{prop:A3} which now follows essentially from Lemma~\ref{lip1}.

\begin{proof}[Proof of Proposition \ref{prop:A3}]
 Fix $t> 0$ and $\ep>0$. Let $\varphi\in \bigcap_{\la\in \La} \big(D(A_\la)\cap C_0(G)\big)$ with $0\leq \varphi\leq 1$, $\varphi(0)=1$ and $\sup_{\la\in \La}\|A_\la\varphi\|_\infty\leq \frac{\varepsilon}{2t}$. Since $\varphi(0)=1$ and
$1-\varphi\in \bigcap_{\la\in \La}D(A_\la)$ with $A_\lambda(1-\varphi)=-A_\lambda \varphi$, it follows from
Lemma \ref{lip1} that
 \begin{align*}
  \big|\big(\SS(t)(1-\varphi)\big)(0)\big|&\leq \|\SS(t)(1-\varphi)-(1-\varphi)\|_\infty\\
  &\leq t\sup_{\la\in \La}\|A_\la \varphi\|_\infty \leq \frac{\ep}{2}.
 \end{align*}
 Note that the estimate in Lemma \ref{lip1} holds for all functions in $\bigcap_{\la\in \La}D(A_\la)$.
Since $\varphi\in C_0(G)$ there exists a compact set $K\subset G$ such that
\[1_{G\sm K}(y)\le 1-\varphi(y)+\frac{\varepsilon}{2}\quad\mbox{for all }y\in G.\]
Hence,
\[
 \EE(1_{G\sm K}(X_t))\leq \EE(1-\varphi(X_t))+\frac{\ep}{2}=|\big(\SS(t)(1-\varphi)\big)(0)|+\frac{\ep}{2}\leq \ep.
\]

In a second step, let $(f_n)$ be a decreasing sequence in $\BUC(G)$ which converges pointwise to $f$. Fix $x\in G$ and $\varepsilon>0$. By the first part, there exists a compact $K\subset G$ such that
$\EE(1_{G\sm K}(X_t))\le\varepsilon/4(1+\|f_1\|_\infty+\|f\|_\infty)$. By Dini's lemma there exists $n_0\in\mathbb{N}$ such that $f_n(x+y)-f(x+y)\le\varepsilon/2$ for all $y\in K$ and $n\ge n_0$. Hence,
\begin{align*}
&\EE\big(f_n(x+X_t)\big)-\EE\big(f(x+X_t)\big)\\
\le & \EE\big((f_n(x+X_t)-f(x+X_t))1_K(X_t)\big)
+\EE\big((f_n(x+X_t)-f(x+X_t))1_{G\sm K}(X_t)\big) \\
\le & \frac{\varepsilon}{2}+2(\|f_1\|_\infty+\|f\|_\infty) \EE(1_{G\sm K}(X_t))\le\varepsilon
\end{align*}
for all $n\ge n_0$. This shows that $f\mapsto \big(\SS(t)f\big)(x)$ is continuous from above on $\BUC(G)$.
\end{proof}

\bibliographystyle{abbrv}

\begin{thebibliography}{20}

%\bibitem{MR1850791}
%P.~Artzner, F.~Delbaen, J.-M. Eber, and D.~Heath.
%\newblock Coherent measures of risk.
%\newblock {\em Math. Finance}, 9(3):203--228, 1999.


%\bibitem{MR2319056}
%P.~Cheridito, H.~M. Soner, N.~Touzi, and N.~Victoir.
%\newblock Second-order backward stochastic differential equations and fully
%  nonlinear parabolic {PDE}s.
%\newblock {\em Comm. Pure Appl. Math.}, 60(7):1081--1110, 2007.

%\bibitem{CHMP}
%F.~Coquet, Y.~Hu, J.~M\'emin, and S.~Peng.
%\newblock Filtration-consistent nonlinear expectations and related
%  g-expectations.
%\newblock {\em Probability Theory and Related Fields}, pages 123:1--27, 2002.


\bibitem{MR2512800}
D.~Applebaum.
\newblock {\emph{L\'evy processes and stochastic calculus}}, volume 116 of
{\emph{Cambridge Studies in Advanced Mathematics}}.
\newblock Cambridge University Press, Cambridge, second edition, 2009.

\bibitem{MR2798103}
W.~Arendt, C.~J.~K. Batty, M.~Hieber, and F.~Neubrander.
\newblock {\emph{Vector-valued {L}aplace transforms and {C}auchy problems}},
  volume~96 of {\emph{Monographs in Mathematics}}.
\newblock Birkh\"auser/Springer Basel AG, Basel, second edition, 2011.

\bibitem{MR1118699}
M.~G.~Crandall, H.~Ishii, and P.-L.~Lions.
\newblock User's guide to viscosity solutions of second order partial
  differential equations.
\newblock {\emph{Bull. Amer. Math. Soc. (N.S.)}}, 27(1):1--67, 1992.


\bibitem{MR1985790}
G.~Da Prato, and J.~Zabczyk.
\newblock {\emph{Second order partial differential equations in {H}ilbert spaces}},
 volume~293 of {\emph{London Mathematical Society Lecture Note Series}}.
\newblock Cambridge University Press, Cambridge, 2002.


\bibitem{MR2754968}
L.~Denis, M.~Hu, and S.~Peng.
\newblock Function spaces and capacity related to a sublinear expectation:
  application to {$G$}-{B}rownian motion paths.
\newblock {\em Potential Anal.}, 34(2):139--161, 2011.

\bibitem{DKN}
R.~Denk, M.~Kupper, and M.~Nendel.
\newblock Kolmogorov type and general extension results for nonlinear expectations.
\newblock \emph{Banach J. Math. Anal.}, 12,  515--540, 2018. 

\bibitem{MR2868935}
Y.~Dolinsky, M.~Nutz, and H.~M. Soner.
\newblock Weak approximation of {$G$}-expectations.
\newblock {\em Stochastic Process. Appl.}, 122(2):664--675, 2012.

\bibitem{MR1721989}
K.-J.~Engel and R.~Nagel.
\newblock {\emph{One-parameter semigroups for linear evolution equations}}, volume
  194 of {\emph{Graduate Texts in Mathematics}}.
\newblock Springer-Verlag, New York, 2000.



\bibitem{MR2229872}
K.-J.~Engel and R.~Nagel.
\newblock {\emph{A short course on operator semigroups}}.
\newblock Universitext. Springer, New York, 2006.

\bibitem{hytonen16}
T.~Hyt\"{o}nen, J.~van Neerven, M.~Veraar, and L.~Weis.
\newblock {\em Analysis in {B}anach Spaces. {V}ol. {I}. {M}artingales and {L}ittlewood-{P}aley Theory}.
\newblock Springer, Cham, 2016.

\bibitem{MR0079232}
G.~A.~Hunt.
\newblock Semi-groups of measures on {L}ie groups.
\newblock {\emph{Trans. Amer. Math. Soc.}}, 81:264--293, 1956.

\bibitem{MR1873235}
N.~Jacob.
\newblock {\emph{Pseudo differential operators and Markov processes. Vol. I}}.
\newblock Imperial College Press, London, 2001.

\bibitem{H2016}
J.~Hollender
\newblock{L\'evy-Type Processes under Uncertainty and Related Nonlocal Equations.} \newblock{\em PhD thesis TU Dresden}, 2016.

\bibitem{kpz}
M.~N.~Kazi-Tani, D.~Possamai and C.~Zhou.
\newblock Second order BSDEs with jumps: formulation and uniqueness.
\newblock {\em The Annals of Applied Probability}, 2015.

\bibitem{kpz1}
M.~N.~Kazi-Tani, D.~Possamai and C.~Zhou.
\newblock Second order BSDEs with jumps: existence and probabilistic representation for fully-nonlinear PIDEs.
\newblock {\em Electronic Journal of Probability}, 2015.

\bibitem{k18}
F.~K{\"u}hn.
\newblock Viscosity solutions to Hamilton-Jacobi-Bellman equations associated with sublinear L\'evy(-type) processes.
\newblock {\em Preprint}, 2018.

\bibitem{Peng}
S.~Peng.
\newblock Nonlinear Expectations and Stochastic Calculus under Uncertainty.
\newblock {\em Preprint}, 2010.

\bibitem{PengHu}
M.~Hu and S.~Peng.
\newblock $G$-L\'evy Processes under Sublinear Expectations.
\newblock {\em Preprint}, 2009.

\bibitem{MR874529}
W.~Linde.
\newblock {\emph{Probability in Banach spaces - stable and infinitely divisible distributions}}.
\newblock John Wiley \& Sons, Ltd., Chichester, 1986.


\bibitem{NutzNeuf}
A.~Neufeld and M.~Nutz.
\newblock Nonlinear L\'evy processes and their characteristics.
\newblock {\em Transactions of the American Mathematical Society}, 2016.

\bibitem{NutzNeuf1}
A.~Neufeld and M.~Nutz.
\newblock Measurability of semimartingale characteristics with respect to the probability law.
\newblock {\em Stochastic Processes and their Applications}, 2014.

\bibitem{Nisio}
M.~Nisio.
\newblock On a non-linear semi-group attached to stochastic optimal control.
\newblock {\em Publ. RIMS, Kyoto Univ.}, 1976.


\bibitem{MR710486}
A.~Pazy.
\newblock {\emph{Semigroups of Linear Operators and Applications to Partial Differential Equations}}, volume~44 of {\emph{Applied Mathematical Sciences}}.
\newblock Springer-Verlag, New York, 1983.


\bibitem{MR2143645}
S.~Peng.
\newblock Nonlinear expectations and nonlinear {M}arkov chains.
\newblock {\em Chinese Ann. Math. Ser. B}, 26(2):159--184, 2005.

\bibitem{PengG}
S.~Peng.
\newblock G-expectation, {G}-{B}rownian motion and related stochastic calculus
  of {I}t\^o type.
\newblock {\em Stochastic Analysis and Applications}, Volume 2 of Abel
  Symposium:541--567, 2007.

\bibitem{MR2474349}
S.~Peng.
\newblock Multi-dimensional {$G$}-{B}rownian motion and related stochastic
  calculus under {$G$}-expectation.
\newblock {\em Stochastic Process. Appl.}, 118(12):2223--2253, 2008.

\bibitem{MR1739520}
K.~Sato.
\newblock {\emph{L\'evy processes and infinitely divisible distributions}},
  volume~68 of {\emph{Cambridge Studies in Advanced Mathematics}}.
\newblock Cambridge University Press, Cambridge, 1999.
\newblock Translated from the 1990 Japanese original, Revised by the author.

\bibitem{stz}
H.~M.~Soner, N.~Touzi and J.~Zhang.
\newblock Wellposedness of second order backward SDEs.
\newblock {\em Probability Theory and Related Fields}, 2012.

%\bibitem{yosida}
%K.~Yosida.
%\newblock {\emph{Functional Analysis}}.
%\newblock Springer-Verlag Berlin Heidelberg, 1995.

%\newblock {\em Stochastic Process. Appl.}, 121(2):265--287, 2011.

%\bibitem{MR2842089}
%H.~M. Soner, N.~Touzi, and J.~Zhang.
%\newblock Quasi-sure stochastic analysis through aggregation.
%\newblock {\em Electron. J. Probab.}, 16:no. 67, 1844--1879, 2011.


\end{thebibliography}

\end{document}